%% file: main.tex
\documentclass[11pt, reqno]{article}
\title{Random homomorphisms and Lipschitz functions on trees}
\author{Alon Heller\thanks{Tel Aviv University. Research supported in part by ISF grant 1361/22.} \and Yinon Spinka\footnotemark[1]}
\date{}

\usepackage{amsmath,amsfonts,amssymb,amsthm}
\usepackage{graphicx,enumerate,mathtools,thmtools}
\usepackage[sorting=nyt, backend=biber, style=numeric, maxbibnames=99]{biblatex}
\usepackage[shortlabels]{enumitem}
\usepackage{xcolor}
\usepackage{forest}
\usepackage{pgfplots}
\pgfplotsset{compat=1.18} 
\usepackage[font=small, labelfont=bf, labelsep=period, margin=0.7cm]{caption}
\usepackage[margin=2.5cm]{geometry}
\usepackage{tocloft}
\usepackage[colorlinks=true, linkcolor=blue, citecolor=blue, urlcolor=blue]{hyperref}

\usepackage{cleveref}
\crefname{theorem}{Theorem}{Theorems}
\Crefname{theorem}{Theorem}{Theorems}
\crefname{lemma}{Lemma}{Lemmas}
\Crefname{lemma}{Lemma}{Lemmas}
\crefname{prop}{Proposition}{Propositions}
\Crefname{prop}{Proposition}{Propositions}
\crefname{claim}{Claim}{Claims}
\Crefname{claim}{Claim}{Claims}
\crefname{cor}{Corollary}{Corollaries}
\Crefname{cor}{Corollary}{Corollaries}
\crefname{question}{Question}{Questions}
\Crefname{question}{Question}{Questions}
\crefname{remark}{Remark}{Remarks}
\Crefname{remark}{Remark}{Remarks}
\Crefname{ssec}{Subsection}{Subsections}

\newtheorem{theorem}{Theorem}[section]
\newtheorem{lemma}[theorem]{Lemma}
\newtheorem{proposition}[theorem]{Proposition}
\newtheorem{claim}[theorem]{Claim}
\newtheorem{corollary}[theorem]{Corollary}

\theoremstyle{definition}
\newtheorem{definition}[theorem]{Definition}
\theoremstyle{remark}
\newtheorem{remark}[theorem]{Remark}
\numberwithin{equation}{section}

\newcommand{\Res}{{\mathcal{R}}}
\newcommand{\RR}{{\mathbb{R}}}
\newcommand{\ZZ}{{\mathbb{Z}}}

\newcommand{\RRnonn}{\left[0,\infty\right)}

\newcommand{\pr}[1]{\left(#1\right)}
\newcommand{\br}[1]{\left[#1\right]}
\newcommand{\cbr}[1]{\left\{#1\right\}}
\newcommand{\abs}[1]{\left|#1\right|}
\newcommand{\floor}[1]{\left\lfloor#1\right\rfloor}

\DeclareMathOperator{\var}{Var}
\newcommand{\prob}[1]{\mathbb{P}\pr{#1}}
\newcommand{\expected}[1]{\mathbb{E}\br{#1}}
\newcommand{\1}{\mathbf{1}}

\begin{document}
\maketitle

\input{abstract}
\tableofcontents

\input{intro}
\input{localized}
\input{delocalized}

\input{lipschitz}
\input{counterexample}

\printbibliography
\end{document}

%% file: abstract.tex
\begin{abstract}
A graph homomorphism is an integer-valued function on the vertex set of a graph that assigns values differing by exactly one to adjacent vertices.
We consider uniformly random homomorphisms on general finite trees, conditioned to take the value zero at all leaves, and study the distribution of the value at the root.
Our main result is a stochastic comparison, both from above and below, between the absolute values of the homomorphism value at the root and certain discrete Gaussian-like random variables.
In particular, we obtain a subgaussian tail bound valid for all deviations, a matching lower bound that holds up to a certain threshold, and upper and lower variance bounds that differ by a constant factor. These bounds depend solely on the effective resistance between the root and the leaves in the associated electrical network.
As a consequence, in the setting of infinite locally finite trees, we obtain that the homomorphism model is localized on transient trees and delocalized on recurrent trees.
Analogous results are obtained for random integer-valued Lipschitz functions.

Our results extend previous results of Benjamini--H{\"a}ggström--Mossel on homomorphisms on regular trees, of Peled--Samotij--Yehudayoff on Lipschitz functions on regular trees, and of Lammers--Toninelli on homomorphisms on trees of minimum degree at least three.

\end{abstract}

%% file: intro.tex
\section{Introduction}

Random height functions are a common framework for probability measures on integer-valued or real-valued functions that assign a height to every vertex of a finite or locally finite graph $G$. One notable family of height functions $\phi:V(G)\to\ZZ$ is graph homomorphisms, given by the condition
\begin{equation*}
\abs{\phi(u)-\phi(v)}=1\quad\forall (u,v)\in E(G).
\end{equation*}
If the graph is bipartite, finite and connected, then fixing the height at one or more vertices leaves finitely many valid functions from which one may be drawn uniformly at random.
Relaxing the constraint to $\abs{\phi(u)-\phi(v)}\leq M$ for adjacent vertices produces the class of integer-valued $M$-Lipschitz functions. Allowing real heights yields an analogous continuous Lipschitz variant.

Random homomorphisms and Lipschitz functions fit into a more general archetype of height function distributions known as gradient Gibbs measures. The defining characteristic of these measures is their dependence on a so-called potential function, which assigns an energy cost to height differences between neighboring vertices. In this paper, we focus mainly on the integer-valued variant. Given a graph $G$ and a potential function $U:\ZZ\to\RR\cup\cbr{\infty}$, the associated distribution on integer-valued height functions $\phi:V(G)\to\ZZ$ is given by
\begin{equation*}
\prob{\phi}\propto \exp\pr{-\sum_{(u,v)\in E(G)} U\pr{\phi(u)-\phi(v)}}
\end{equation*}
which is properly defined when $G$ is finite and the heights are fixed for a subset of vertices in a suitable way, a choice known as a boundary condition.
As mentioned, random homomorphisms and Lipschitz functions may be seen as instances of this framework. In the homomorphism model, the potential is given by $U(x)=0$ for $x=\pm1$ and $U(x)=\infty$ otherwise, enforcing the constraint that neighboring heights differ by exactly one. In the $M$-Lipschitz model, the potential is defined by $U(x)=0$ for $\abs{x}\leq M$ and $U(x)=\infty$ otherwise. Both of these are examples of hard-constrained models, where some height differentials are strictly forbidden.

An important and well-known soft-constrained example is the integer-valued Gaussian free field (GFF), corresponding to the quadratic potential $U(x)=\beta x^2$. This model penalizes large height differentials rather than forbidding them. Another canonical example of the soft-constrained family of models is the so-called solid-on-solid (SOS) model, whose potential is given by $U(x)=\beta \abs{x}$.

The framework for gradient Gibbs measures extends naturally to real-valued height functions. Canonical examples include the aforementioned real-valued $M$-Lipschitz model, and the (real-valued) Gaussian free field for which the potential is once again given by $U(x)=\beta x^2$.

In the setting of locally finite graphs, a common approach is to set boundary conditions at increasing distances from a fixed vertex and examine the distribution of its height, typically fixing the height at the boundary to be zero. This approach is central to the study of localization versus delocalization:
Localization occurs when this family of height distributions is tight as the boundary varies, while delocalization occurs when the probability of the height belonging to any bounded set vanishes as the boundary recedes to infinity.
Informally, the typical height remains bounded in localized models, whereas in delocalized models the height fluctuations grow to infinity.

In this paper, we study integer-valued height functions on trees under zero boundary conditions.
Our results provide two-sided bounds on the distribution of the height at the root, both for random graph homomorphisms and for random integer-valued $M$-Lipschitz functions. In each case, these bounds are expressed in terms of the effective resistance between the root and the boundary in the associated electrical network (see, e.g., \cite[Chapter 2]{LP17} for relevant background). For locally finite trees, we deduce that both models exhibit localization or delocalization depending solely on whether the underlying tree is transient or recurrent (for simple random walk), respectively.

Our first main result is a complete dichotomy determining localization or delocalization of random homomorphisms on any locally finite tree.

\begin{theorem} \label{thm:equivalence_hom}
Let $T$ be a locally finite tree. Then homomorphisms on $T$ are localized if $T$ is transient, and delocalized otherwise.
\end{theorem}

We further obtain a refinement of \Cref{thm:equivalence_hom}, giving matching upper and lower distribution bounds on the height at the root under zero boundary conditions. Due to the parity conditions imposed by homomorphisms, we introduce the notion of \textbf{even rooted trees}: Trees in which every leaf is at an even distance from the root. The following result is given for even trees, but a similar result easily follows for \emph{odd} rooted trees as well.

In the following, we denote by
$\mathcal{N}_I \pr{0,\sigma^2}$ a discrete Gaussian variable restricted to the subset $\emptyset \neq I \subseteq \ZZ$, taking value $x \in I$ with probability proportional to $e^{-x^2/2\sigma^2}$.

\begin{theorem} \label{thm:main_result_hom}
Let $(T,v^*)$ be an even finite rooted tree, and let $\phi:V(T)\to\ZZ$ be a uniformly random homomorphism satisfying $\phi\equiv0$ on the leaves of $T$. Let $R$ be the effective resistance between the root $v^*$ and the leaves. Then the distribution of the height at the root, $\phi\pr{v^*}$, satisfies
\begin{equation*}
\abs{\mathcal{N}_{2\ZZ\cap(-R/2,R/2)}\pr{0,R/4}} \,\le_{st}\, \abs{\phi\pr{v^*}} \,\le_{st}\, \abs{\mathcal{N}_{2\ZZ} \pr{0,4R}} .
\end{equation*}
In particular, its variance is bounded by
\begin{equation*}
\floor{\frac{R}{4e^4\sqrt{\pi}}} \leq\var\pr{\phi\pr{v^*}}\leq 4R .
\end{equation*}
\end{theorem}
Additional information about the distribution of $\phi\pr{v^*}$ is obtained in \Cref{sec:hom-localization,sec:hom-delocalization}. We also mention here that our results apply under more general boundary conditions, see \Cref{remark:hom_loc_different_boundary_conds,remark:hom_deloc_different_boundary_conds}.

\smallskip

We obtain analogous results for integer-valued Lipschitz functions.

\begin{theorem} \label{thm:equivalence_lip}
Let $T$ be a locally finite tree and let $M$ be a positive integer. Then integer-valued $M$-Lipschitz functions on $T$ are localized if $T$ is transient, and delocalized otherwise.
\end{theorem}

We also provide an analogue of \Cref{thm:main_result_hom} in the Lipschitz setting, see \Cref{cor:upper_bounds_on_lip_root_laws,cor:lower_bounds_on_lip_root_laws}.

\subsection{Background and discussion}
The study of gradient Gibbs measures has received significant attention in recent years. The real-valued Gaussian free field (GFF) is well-understood; it is closely linked to electrical resistance and random walks, see for example \cite[Theorem 2.3 and Proposition 2.24]{LP17}. In particular, it is known to localize on transient graphs and delocalize on recurrent graphs. The integer-valued Gaussian free field has a more complex behavior, dependent on its so-called inverse temperature. However, because its variance is strictly bounded by that of the real-valued GFF via Gaussian domination \cite{FSS76, FS81, She05}, the integer-valued GFF is similarly known to localize on any transient graph regardless of temperature.

A recent result by Sellke \cite{Sel24} regarding continuous models extends localization on transient graphs to a wide variety of soft-constraint potential functions, including the solid-on-solid (SOS) model. The boundary behavior of these gradient models on trees has also been a subject of recent focus; for instance, Henning and K{\"u}lske \cite{HK24} recently detailed the boundary pinning mechanics for general nearest-neighbor integer-valued gradient models on regular trees.

Comparatively less is known in the hard-constraint setting, although recent work has begun to explore the behavior of hard-core constraints and SOS models with countable spin values specifically on Cayley trees \cite{RH24}. Random homomorphisms and integer-valued Lipschitz functions have been shown to localize on some specific graph classes, including expanders \cite{PSY13exp,KLP24,DJ26}, hypercubes \cite{Kah01,Gal03} and high dimensional lattices \cite{Pel17, PS20}.

On regular trees, localization was established first for random homomorphisms by Benjamini, H{\"a}ggström and Mossel~\cite{BHM00}, and then for (both discrete and continuous) random Lipschitz functions by Peled, Samotij and Yehudayoff~\cite{PSY13lip}. In both instances, the authors utilized a recursive approach to obtain doubly exponential tail bounds on the root height distribution. This line of inquiry remains highly active; recently, Butler, Krishnan, Ray and Spinka~\cite{BKRS24} provided an alternative proof of localization for 1-Lipschitz functions on regular trees, and established a phase transition in the weak local convergence of these measures based on the degree of the tree.

Recently, Lammers and Toninelli \cite{LT24} extended the localization of homomorphisms to any tree in which every vertex has at least two children. They proved by induction that the root height distribution in such trees is $\alpha$-strongly log-concave for a universal constant $\alpha>1$ (see \Cref{sec:proof-outline} below for an explanation of strong log-concavity and other notions discussed here). In this paper, we build upon this argument and extend it to general transient trees, by establishing strong log-concavity with a non-constant parameter $\alpha$ depending on the tree and vertex. We also construct similar arguments to prove delocalization on recurrent trees, as well as extend the results to Lipschitz functions. To tailor the argument accordingly, we introduce limited log-concavity, as well as another variant of strong log-concavity. Notably, we additionally establish localization for continuous real-valued Lipschitz functions on transient trees via a limiting argument, though the behavior of the continuous model on recurrent trees remains an open question.

A natural question that arises from these results is whether the transience-recurrence dichotomy holds for general graphs beyond trees. In \Cref{sec:counterexample} we show that this is not the case. Namely, we present a transient locally finite graph on which random homomorphisms are delocalized, as well as a recurrent one on which homomorphisms are localized. These counterexamples also rule out such a dichotomy for random Lipschitz functions.

\subsection{Proof outline}\label{sec:proof-outline}
We begin with some notation.
Let $\pr{T,v^*}$ be a finite or locally finite rooted tree. Given a vertex $v\in T$, we denote the set of children of $v$ by $N(v)\subseteq V(T)$. We denote by $T[v]$ the induced tree on the set of descendants of $v$, rooted in $v$. We also use basic notions from the theory of electrical networks on graphs (see, e.g., \cite[Chapter 2]{LP17}). All our networks are assumed to consist solely of edges of resistance 1.
Given a finite rooted tree $\pr{T,v^*}$, we denote by $\Res_T$ the effective resistance between the root and the leaves in the associated electrical network. Likewise, given an infinite, locally finite rooted tree, we denote by $\Res_T$ the effective resistance between the root and infinity. It is well known that $\Res_T<\infty$ if and only if $T$ is transient (for the simple random walk).

To understand the motivation behind the variance-resistance connection,
it is useful to first consider the (real-valued) Gaussian free field model. Let $\pr{T,v^*}$ be a finite rooted tree. The density of the law of the root height $\phi\pr{v^*}$ (with respect to Lebesgue measure), denoted $f_{T,\mathrm{GFF}}:\RR\to\RRnonn$, satisfies the recursion
\begin{equation} \label{eq:gff_law_flow}
f_{T,\mathrm{GFF}} \propto \prod_{v\in N\pr{v^*}} \pr{f_{T[v],\mathrm{GFF}} * \varphi},
\end{equation}
where $\varphi(x)=\frac{1}{\sqrt{2\pi}}e^{-x^2/2}$. Since all operations involved are normalized products or convolutions with Gaussians, one deduces by induction that $\phi\pr{v^*}$ is Gaussian. It turns out that the variance of this Gaussian variable equals the effective resistance between the root and the leaves. Indeed, suppose that $f_{T[v],\mathrm{GFF}}(x) \propto e^{-x^2/2\sigma_v^2}$ for each child $v$ of the root. Then \eqref{eq:gff_law_flow} gives $f_{T,\mathrm{GFF}}(x) \propto e^{-x^2/2\sigma^2}$ where
\begin{equation} \label{eq:gff_sigma_flow}
\frac{1}{\sigma^2} = \sum_{v\in N\pr{v^*}} \frac{1}{\sigma_v^2 + 1}.
\end{equation}
On the other hand, as $T$ is composed of parallel copies of $T[v]$, each separated from the root by one additional edge of resistance 1, elementary electrical laws give that
\begin{equation} \label{eq:gff_resistance_flow}
\frac{1}{\Res_T} = \sum_{v\in N\pr{v^*}} \frac{1}{\Res_{T[v]} + 1}.
\end{equation}
Combining the two recursive formulas in \eqref{eq:gff_sigma_flow} and \eqref{eq:gff_resistance_flow}, we deduce by induction that $\sigma^2=\Res_T$ and
\begin{equation*}
\phi\pr{v^*}\sim \mathcal{N}\pr{0,\Res_T}.
\end{equation*}

We now consider the case of homomorphisms. Since the distribution at the root follows a similar recursive formula as in~\eqref{eq:gff_law_flow}, with the kernel $\1_{\cbr{-1,1}}$ replacing $\varphi$ (see \Cref{lem:hom_law_flow}), one may expect a similar behavior. Quantifying this behavior presents a challenge, as the laws are no longer Gaussian, and tracking the variance alone does not provide sufficient control. To this end, we introduce notions of strong log-concavity and limited log-concavity, measures that provide upper and lower bounds on the discrete log-curvature.
Here, by the \emph{discrete log-curvature} of an integer-valued probability law $f:2\ZZ\to\RRnonn$ or $f:2\ZZ+1\to\RRnonn$, we refer to the quantity $\frac{f(x-2)f(x+2)}{f(x)^2}$, which can be thought of as a discrete analogue of the second log-derivative.
Notably, (discrete) Gaussian variables are defined by a constant (discrete) log-curvature $e^{-1/\sigma^2}$ throughout their domain.

Let us begin with the localization argument, which is slightly simpler.
We say that $f$ is \textbf{\mbox{$\alpha$-strongly} log-concave} for $\alpha > 1$ if its support is a contiguous interval of even or odd integers and
\begin{equation}\label{eq:alpha-strongly-log-concave}
f(n)^2 \geq \alpha\cdot f(n-2)f(n+2) \qquad\text{for all } n.
\end{equation}
In words, this means that the discrete log-curvature of $f$ is bounded above by $1/\alpha$.
Mimicking the recursive approach above, we examine the propagation of this quantity up the tree (see \Cref{fig:flow}). Naturally, this requires studying its interaction with convolution by $\1_{\cbr{-1,1}}$ and with products. We find that this quantity propagates comparably to the effective resistance. Specifically, we shall see that if the root height law in the subtree $T[v]$ is $\alpha_v$-strongly log-concave for each child $v$ of the root, then the law in $T$ is $\alpha_{v^*}$-strongly log-concave where
\begin{equation} \label{eq:hom_alpha_flow}
\alpha_{v^*}=\prod_{v\in N\pr{v^*}}\frac{4}{\pr{\alpha_v^{-1}+1}^2}.
\end{equation}
This rule turns out to be roughly comparable to \eqref{eq:gff_resistance_flow} near $\alpha=1$ via the relation $\alpha_{v^*}=e^{1/\Res_T}$ (and $\alpha_v=e^{1/\Res_{T[v]}}$).
In particular, the quantities $(\alpha_v)_{v \in T}$ propagate along the tree in essentially the same way that the resistances $(\Res_{T[v]})_{v \in T}$ do, which leads to the conclusion that the root height law is $\alpha_{v^*}$-strongly log-concave with $\alpha_{v^*} = e^{1/\Res_T}$. Since strong log-concavity implies stochastic dominance by some discrete Gaussian, we obtain our upper bounds on the height at the root. In particular, since transient trees have finite resistance, this yields the localization result for transient trees.
\input{treefig}

Let us now discuss the delocalization argument, which follows a similar approach, but is slightly more technical. We say that a log-concave function $f$ is \textbf{$\pr{R,\gamma}$-limited log-concave} if is satisfies
\begin{equation*}
f(n)^2 \leq e^{\gamma/(R-|n|)}\cdot f(n-2)f(n+2) \qquad\text{whenever }|n|<R.
\end{equation*}
In contrast to strong log-concavity, which imposes a uniform upper bound on the log-curvature, limited log-concavity imposes a lower bound within a restricted interval around zero, which varies across the interval.
Such a restriction is necessary since the root height law is finitely supported, so any positive global lower bound on the log-curvature would fail to hold. To help accommodate this, the lower bound is designed to weaken as $\abs{n}$ increases, smoothly transitioning from a strong constraint near $n=0$ to no constraint at all at $n=\pm R$.

With this more involved notion of limited log-concavity, it turns out that the propagation of the $R$ parameter follows the same rule as the resistance in~\eqref{eq:gff_resistance_flow}. Specifically, we shall see that if $\gamma \geq 8$ and the root height law in the subtree $T[v]$ is $\pr{R_v,\gamma}$-limited log-concave for each child $v$ of the root, then the root height law in $T$ is $\pr{R_{v^*},\gamma}$-limited log-concave where
\begin{equation} \label{eq:hom_R_flow}
\frac{1}{R_{v^*}} = \sum_{v\in N\pr{v^*}} \frac{1}{R_v + 1}.
\end{equation}
In particular, this leads to the conclusion that the root height law is $\pr{\Res_T,\gamma}$-limited log-concave. Since limited log-concavity implies stochastic dominance of a certain truncated discrete Gaussian, we obtain our lower bounds on the height at the root. In particular, since recurrent trees have infinite resistance, this yields the delocalization result for recurrent trees.

The high-level proof idea for Lipschitz functions is nearly identical. The main obstacle is that the kernel $\1_{\cbr{-M,...,M}}$ does not possess the classical property of $\alpha$-strong log-concavity (analogous to~\eqref{eq:alpha-strongly-log-concave}) due to its flat shape. In particular, the point mass $\1_{\cbr{0}}$ already loses its strong log-concavity after a single convolution with this kernel. For this reason we introduce \textbf{$(\alpha,m)$-strong log-concavity}, designed so that the indicator function of an interval of length $m$ is strongly log-concave (here we set $m=2M+1$).
Namely, we require that $f$ is weakly log-concave and that
\begin{equation}
f(n+1)f(n+m) \geq \alpha\cdot f(n)f(n+m+1) \qquad\text{ for all }n \in \ZZ.
\end{equation}
As opposed to $\alpha$-strong log-concavity, which imposes a uniform upper bound of $\frac{1}{\alpha}$ on the log-curvature, $(\alpha,m)$-strong log-concavity only requires a uniform upper bound of $1$ (weak log-concavity) combined with a stricter upper bound on the average log-curvature over intervals of length $m$.
Using this modified version of log-concavity, the inductive argument continues mostly unchanged, with the propagation of the parameter $\alpha$ in this case following the modified rule
\begin{equation} \label{eq:lip_alpha_flow}
\alpha_{v^*}=\prod_{v\in N\pr{v^*}}\frac{2M+1}{2M+\alpha_v^{-2M-1}}.
\end{equation}
This rule implies a relation between $\alpha_{v*}$ and $\Res_T$, which leads to the localization result on transient trees.
The delocalization argument for Lipschitz functions is based on a similar approach as for homomorphisms.

\paragraph{Organization.}
In \Cref{sec:localization}, we develop useful properties of log-concavity and use them to derive bounds on the root height distribution, which in turn establish localization of homomorphisms on transient trees.
In \Cref{sec:delocalization}, we introduce limited log-concavity and use it to prove distribution bounds that lead to delocalization of homomorphisms on recurrent trees, thus completing the proof of \Cref{thm:equivalence_hom,thm:main_result_hom}.
In \Cref{sec:lipschitz} we prove \Cref{thm:equivalence_lip}, extending these results to integer-valued Lipschitz functions. We also present a partial analogue for the real-valued Lipschitz model.
In \Cref{sec:counterexample} we present counterexamples showing that the transient-recurrent dichotomy result does not hold in general for locally finite graphs.

%% file: treefig.tex
\begin{figure}[t]
\centering
\scalebox{0.75}{
\begin{forest}
for tree={
  draw, circle, minimum size=1.2em, inner sep=1pt,
  fill=blue!20,
  edge={-latex},
  s sep=18mm,
  l sep=16mm,
  label position=right, label distance=2pt,
}
[A, label={right:$R_A=\frac{1}{\frac{1}{R_B+1}+\frac{1}{R_C+1}}$}
  [B, label={right:$R_B=\frac{1}{\frac{1}{R_D+1}+\frac{1}{R_E+1}}$}
    [D, label={right:$R_D$}, tikz+={
            \draw[dotted] () -- +(-0.4,-0.6);
            \draw[dotted] () -- +(0.4,-0.6);} ]
    [E, label={right:$R_E$}, tikz+={
            \draw[dotted] () -- +(-0.4,-0.6);
            \draw[dotted] () -- +(0.4,-0.6);} ]
  ]
  [C, label={right:$R_C=\frac{1}{\frac{1}{R_F+1}+\frac{1}{R_G+1}}$}
    [F, label={right:$R_F$}, tikz+={
            \draw[dotted] () -- +(-0.4,-0.6);
            \draw[dotted] () -- +(0.4,-0.6);} ]
    [G, label={right:$R_G$}, tikz+={
            \draw[dotted] () -- +(-0.4,-0.6);
            \draw[dotted] () -- +(0.4,-0.6);} ]
  ]
]
\end{forest}
\begin{forest}
for tree={
  draw, circle, minimum size=1.2em, inner sep=1pt,
  fill=blue!20,
  edge={-latex},
  s sep=18mm,
  l sep=16mm,
  label position=right, label distance=2pt,
}
[A, label={right:$\alpha_A=\Delta\pr{\alpha_B}\cdot\Delta\pr{\alpha_C}$}
  [B, label={right:$\alpha_B=\Delta\pr{\alpha_D}\cdot\Delta\pr{\alpha_E}$}
    [D, label={right:$\alpha_D$}, tikz+={
            \draw[dotted] () -- +(-0.4,-0.6);
            \draw[dotted] () -- +(0.4,-0.6);} ]
    [E, label={right:$\alpha_E$}, tikz+={
            \draw[dotted] () -- +(-0.4,-0.6);
            \draw[dotted] () -- +(0.4,-0.6);} ]
  ]
  [C, label={right:$\alpha_C=\Delta\pr{\alpha_F}\cdot\Delta\pr{\alpha_G}$}
    [F, label={right:$\alpha_F$}, tikz+={
            \draw[dotted] () -- +(-0.4,-0.6);
            \draw[dotted] () -- +(0.4,-0.6);} ]
    [G, label={right:$\alpha_G$}, tikz+={
            \draw[dotted] () -- +(-0.4,-0.6);
            \draw[dotted] () -- +(0.4,-0.6);} ]
  ]
]
\end{forest}
}
\caption{\emph{Left:} The effective resistance between the root and the leaves in various subtrees
of a finite rooted tree $T$. The resistance for each subtree is written next to its root. \emph{Right:} Parameters $\alpha$ for which the law of the root height for a random homomorphism is $\alpha$-strong log concave, written for the same subtrees. Here, $\Delta(\alpha)=4\pr{\alpha^{-1}+1}^{-2}$.}
\label{fig:flow}
\end{figure}
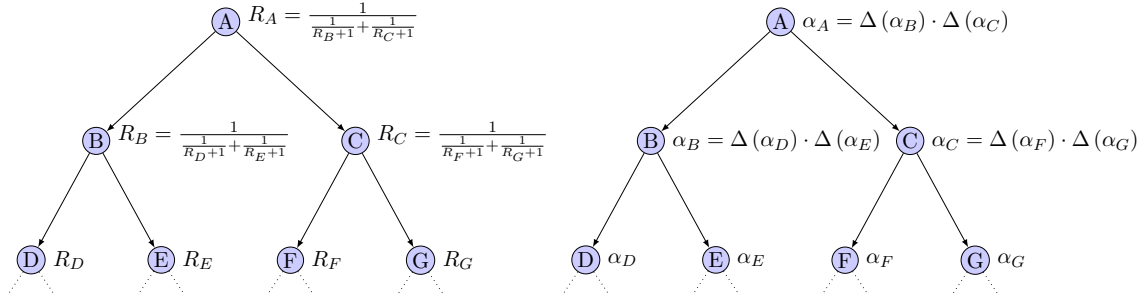

%% file: localized.tex
\section{Localization of homomorphisms on transient trees} \label{sec:localization}

In this section we prove the localization part of \Cref{thm:equivalence_hom} (\Cref{thm:hom_localization} below) and the upper bounds in \Cref{thm:main_result_hom} (\Cref{cor:upper_bounds_on_hom_root_laws} below). 
We achieve this by showing that $\alpha$-strong log-concavity propagates in a comparable way to the effective resistance, as explained in the proof outline above.

Given a finite odd or even rooted tree $(T,v^*)$, we denote by $f_{T, \mathrm{Hom}}$ the law of $\phi(v^*)$, where $\phi:T\to\ZZ$ is a uniformly random homomorphism vanishing on the set of leaves. Recall that for a vertex $v$, $N(v)$ denotes its set of children, $T[v]$ is the subtree rooted at $v$, and $\mathcal{R}_T$ is the effective resistance between the root and the leaves in $T$.

Note the parity condition imposed by homomorphisms: The height at the root has the same parity as its distance from the leaves. Consequently, $f_{T, \mathrm{Hom}}$ has domain $2\ZZ$ or $2\ZZ+1$ when $T$ is an even or odd tree, respectively.

\subsection{Tree recursion}

The core argument connecting the properties of $f_{T, \mathrm{Hom}}$ to the effective resistance revolves around the recursive nature of both, as is reflected by the following two simple lemmas.

\begin{lemma} \label{lem:hom_law_flow}
Let $(T,v^*)$ be a nontrivial odd or even rooted tree. Then
\begin{equation*}
f_{T,\mathrm{Hom}} \propto \prod_{v\in N\pr{v^*}} \pr{f_{T[v],\mathrm{Hom}} * \1_{\cbr{-1,1}}} .
\end{equation*}
\end{lemma}

\begin{proof}
Fix $n\in\ZZ$. We count the number of homomorphisms assigning the leaves height 0 and the root height $n$. This quantity is proportional to $f_{T,\mathrm{Hom}}(n)$. On the other hand, by considering all possible choices of assigning each child $v\in N\pr{v^*}$ a height $n\pm 1$, it is also proportional to the sum of the expression
\begin{equation*}
\prod_{v\in N\pr{v^*}} \pr{f_{T[v],\mathrm{Hom}}\pr{n\pm 1}}
\end{equation*}
taken over all possible sign choices, which is precisely $\prod_{v\in N\pr{v^*}} \pr{f_{T[v],\mathrm{Hom}} * \1_{\cbr{-1,1}}} (n)$.
\end{proof}

\begin{lemma} \label{lem:resistance_flow}
Let $(T,v^*)$ be a nontrivial odd or even rooted tree. Then
\begin{equation*}
\frac{1}{\mathcal{R}_T}
= \sum_{v\in N\pr{v^*}} \frac{1}{\mathcal{R}_{T[v]}+1} .
\end{equation*}
\end{lemma}

\begin{proof}
Observe that $T$ is composed of copies of $T[v]$ for $v\in N\pr{v^*}$ connected in parallel to $v^*$, each separated from it by an additional edge of resistance 1. The result follows from the elementary series and parallel laws.
\end{proof}

\subsection{Log-concavity}
\begin{definition}
Let $f:Z\to\RRnonn$, where $Z=2\ZZ$ or $Z=2\ZZ+1$. We say that $f$ is \emph{(weakly) log-concave} if its support consists of a contiguous sequence of even or odd integers, and
\begin{equation*}
f(n)^2 \geq f(n-2)f(n+2) \qquad\text{for all } n \in Z.
\end{equation*}
We further say that $f$ is \emph{$\alpha$-strongly log-concave} for $\alpha > 1$ if
\begin{equation*}
f(n)^2 \geq \alpha\cdot f(n-2)f(n+2) \qquad\text{for all } n \in Z.
\end{equation*}
\end{definition}

Note that these are simply uniform upper bounds on the log-curvature of $f$ where it is well-defined. As we show in \Cref{thm:hom_root_law_is_strong_log_conc}, the law $f_{T, \mathrm{Hom}}$ is $\alpha$-strongly log-concave for every tree $T$, where $\alpha$ only depends on the effective resistance between the root and the leaves. Following \Cref{lem:hom_law_flow}, the building blocks of our inductive argument are interactions between strong log-concavity and convolution by $\1_{\cbr{-1,1}}$ or products of probability laws.

\begin{proposition} \label{prop:strong_log_conc_multiplicativity}
Let $f_1,f_2:Z\to\RRnonn$ such that $f_i$ is $\alpha_i$-strongly log-concave, where $Z=2\ZZ$ or $Z=2\ZZ+1$.
Then $f_1\cdot f_2$ is $\pr{\alpha_1\alpha_2}$-strongly log-concave.
\end{proposition}

\begin{proof}
Denote $g\coloneqq f_1\cdot f_2$. For any $n\in Z$, we have
\begin{align*}
g(n)^2 &=f_1(n)^2\cdot f_2(n)^2 \\
 &\leq \alpha_1\alpha_2
\cdot f_1 (n-2) f_1 (n+2) \cdot f_2 (n-2) f_2 (n+2) \\
&= \alpha_1 \alpha_2 \cdot g(n-2) g(n+2) .
\end{align*}
Moreover, $g$ has contiguous support, which is the intersection of the supports of $f_1$ and $f_2$. Therefore $g$ is $\pr{\alpha_1\alpha_2}$-strongly log-concave.
\end{proof}

\begin{proposition} \label{prop:conv_strong_log_conc}
Let $f:Z\to\RRnonn$ be $\alpha$-strongly log-concave, where $Z=2\ZZ$ or $Z=2\ZZ+1$. Then $f*\1_{\cbr{-1,1}}$ is $\beta$-strongly log-concave for $\beta=4\pr{\alpha^{-1}+1}^{-2}$.
\end{proposition}

\begin{proof}
Denote $g\coloneqq f*\1_{\cbr{-1,1}}$. We need to prove that $g(n)^2\geq \beta \cdot g(n-2)g(n+2)$ for all $n\in Z+1$, or equivalently that
\begin{equation*}
\pr{f(n-1)+f(n+1)}^2 \geq \beta \pr{f(n-3)+f(n-1)}\pr{f(n+1)+f(n+3)}.
\end{equation*}
We may assume that $f(n-1)$ and $f(n+1)$ are positive, otherwise the right-hand must vanish since $f$ has a contiguous support. Therefore $f(n-3)\leq \frac{f(n-1)^2}{\alpha f(n+1)}$ and $f(n+3)\leq \frac{f(n+1)^2}{\alpha f(n-1)}$ by strong log-concavity of $f$. It thus suffices to prove that
\begin{equation*}
\pr{f(n-1)+f(n+1)}^2 \geq \beta \pr{\frac{f(n-1)^2}{\alpha f(n+1)}+f(n-1)}\pr{f(n+1)+\frac{f(n+1)^2}{\alpha f(n-1)}}.
\end{equation*}
Writing $t\coloneqq \frac{f(n+1)}{f(n-1)}$ and dividing by $f(n-1)^2$, this is equivalent to
\begin{equation} \label{ineq:strong_log_conc_abt_condition}
\pr{1+t}^2 \geq \beta \pr{\frac{1}{\alpha}+t}\pr{1+\frac{t}{\alpha}} .
\end{equation}
Denote $F(x)\coloneqq\pr{1+x}^2{\pr{\frac{1}{\alpha}+x}}^{-1}{\pr{1+\frac{x}{\alpha}}}^{-1}$. Observe that $F$ is minimized at $x=1$, as its log-derivative is
\begin{equation*}
\frac{d}{dx} \log F(x) =
\frac{2}{1+x} - \frac{1}{\frac{1}{\alpha}+x} - \frac{1}{\alpha+x}=
\frac{(\alpha-1)^2(x-1)}{\pr{1+\alpha x}\pr{\alpha+x}(1+x)},
\end{equation*}
which is negative on $(0,1)$ and positive on $(1,\infty)$. Therefore $F(t)\geq F(1) = \beta$, yielding \eqref{ineq:strong_log_conc_abt_condition}.

Finally, the support of $g$ is a contiguous sequence of integers of the same parity. Namely, if the support of $f$ is $[a,b]\cap Z$ then the support of $g$ is $[a-1,b+1]\cap(Z+1)$. Therefore $g$ is $\beta$-strong log-concave.
\end{proof}

We quickly note that the same closure properites hold for weak log-concavity as well, for example by replaying the proofs of \Cref{prop:strong_log_conc_multiplicativity,prop:conv_strong_log_conc} with $\alpha=1$.

\begin{proposition} \label{prop:weak_log_conc_closure}
The family of weakly log-concave functions is closed under products and under convolution by $1_{\cbr{-1,1}}$.
\end{proposition}

In order to later leverage the strong log-concavity of the root height law $f_{T,\mathrm{Hom}}$, we provide distribution bounds on strongly log-concave laws.
The following lemma, which is a particular case of \cite[Theorem~1.C.1]{SS07}, helps to establish such a bound in the form of stochastic dominance.

\begin{lemma} \label{lem:likelyhood_dom_implies_stoch_dom}
Let $X,Y$ be integer-valued random variables with laws $f_X,f_Y:\ZZ\to[0,1]$ respectively. Suppose that $f_Y\pr{n}/f_X\pr{n}$ is increasing over the union of the supports of $X$ and $Y$ (where $a/0$ is taken to be equal to $\infty$ whenever $a>0$).
Then $Y$ stochastically dominates $X$.
\end{lemma}

For technical reasons, we only derive bounds for even-valued distributions. Similar, yet slightly different bounds hold in the odd-valued case.

\begin{proposition} \label{prop:bounds_on_strong_log_conc_laws}
Let $X$ be a $2\ZZ$-valued random variable, whose law is symmetric around zero and $\alpha$-strongly log-concave. Then the following hold:
\begin{enumerate}[(a)]
    \item $\abs{X}$ is stochastically dominated by $\abs{\mathcal{N}_{2\ZZ} \pr{0,\frac{4}{\log\pr{\alpha}}}}$; \label{item:strong_stoch_dom}
    \item $\var{X} \leq \frac{4}{\log(\alpha)}$; \label{item:strong_variance_bound}
    \item $\prob{X = 2n}\leq \alpha^{-n^2/2}\cdot\prob{X = 0}$ for all $n\in\ZZ$. \label{item:strong_pointwise_bound}
\end{enumerate}
\end{proposition}
\begin{proof}
Let $f:2\ZZ\to[0,1]$ be the law of $X$. Towards showing \ref{item:strong_stoch_dom}, and motivated by \Cref{lem:likelyhood_dom_implies_stoch_dom}, we aim to upper bound the successive quotients $f(2n+2)/f(2n)$. By symmetry and $\alpha$-strong log-concavity of $f$, we have $f(0)^2 \geq \alpha f(-2)f(2) = \alpha f(2)^2$ so that
\begin{equation*}
f(2) \leq \alpha^{-1/2} f(0).
\end{equation*}
Furthermore, when $2n\geq0$ is in the support of $f$ we have $f(2n)^2 \geq \alpha f(2n-2)f(2n+2)$, so that
\begin{equation*}
\frac{f(2n+2)}{f(2n)} \leq \frac{\alpha^{-1}f(2n)}{f(2n-2)} .
\end{equation*}
Repeated application of this inequality yields
\begin{equation} \label{ineq:localization_stoch_dom_helper_ineqs}
\frac{f(2n+2)}{f(2n)} \leq \alpha^{-n-1/2} .
\end{equation}
We wish to prove that $\abs{X}$ is stochastically dominated by $\abs{Y}$ where $Y=\mathcal{N}_{2\ZZ} \pr{0,4/\log\pr{\alpha}}$.
The laws of $\abs{X}$ and $\abs{Y}$ are respectively given by
\begin{equation*}
f_{\abs{X}}(2n)=\begin{cases}
			f(0), & n=0\\
            2f(2n), & n > 0
		 \end{cases}
\end{equation*}
and
\begin{equation*}
f_{\abs{Y}}(2n)\propto\begin{cases}
			1, & n=0\\
            2\alpha^{-n^2/2}, & n > 0
		 \end{cases}.
\end{equation*}
By \eqref{ineq:localization_stoch_dom_helper_ineqs}, for $n>0$ we have
\begin{equation*}
\frac{f_{\abs{X}}\pr{2n+2}}{f_{\abs{X}}\pr{2n}} \leq \alpha^{-n-1/2} = \frac{f_{\abs{Y}}\pr{2n+2}}{f_{\abs{Y}}\pr{2n}} ,
\end{equation*}
while for $n=0$ we have
\begin{equation*}
\frac{f_{\abs{X}}\pr{2}}{f_{\abs{X}}\pr{0}} \leq 2\alpha^{-1/2} = \frac{f_{\abs{Y}}\pr{2}}{f_{\abs{Y}}\pr{0}}.
\end{equation*}
By \Cref{lem:likelyhood_dom_implies_stoch_dom}, $\abs{Y}$ stochastically dominates $\abs{X}$, which establishes \ref{item:strong_stoch_dom}.
It is known that the variance of an \emph{integer-valued} Gaussian around zero is bounded above by that of a corresponding \emph{continuous} Gaussian, that is, $\var\mathcal{N}_{\ZZ}(0,\sigma^2) \le \var\mathcal{N}(0,\sigma^2) = \sigma^2$ (see, e.g., \cite{CKS20}). By the stochastic dominance we have just established, we get
\begin{equation*}
\var{X} \leq \var{Y}=4\var\mathcal{N}_{\ZZ}\pr{0,\frac{1}{\log\pr{\alpha}}}
\leq \frac{4}{\log\pr{\alpha}},
\end{equation*}
which proves \ref{item:strong_variance_bound}. Finally, for \ref{item:strong_pointwise_bound} observe that  \eqref{ineq:localization_stoch_dom_helper_ineqs} gives
\begin{equation*}
\frac{f(2n)}{f(0)} = \frac{f(2n)}{f(2n-2)} \cdot \cdots \cdot \frac{f(2)}{f(0)}
\leq \alpha^{-(2n-1)/2} \cdots \alpha^{-1/2} = \alpha^{-n^2 / 2}
\end{equation*}
for $n>0$, and by symmetry the same is true for $n < 0$.
\end{proof}

\subsection{Proof of localization}\label{sec:hom-localization}
We are now ready to construct the complete inductive argument.
\begin{theorem} \label{thm:hom_root_law_is_strong_log_conc}
Let $(T,v^*)$ be an odd or even finite rooted tree. Then $f_{T,\mathrm{Hom}}$ is $e^{1/R}$-strongly log-concave for $R=\mathcal{R}_T$.
\end{theorem}
\begin{proof}
We induct on the depth of the tree. For the trivial tree $T$ of depth zero, $f_{T,\mathrm{Hom}}=\1_{\cbr{0}}$ is $\infty$-strongly log-concave (that is, $\alpha$-strongly log-concave for all $\alpha>1$).

Suppose that $T$ is a nontrivial tree. Applying the induction hypothesis to $T[v]$ for $v\in N\pr{v^*}$, we obtain that $f_{T[v],\mathrm{Hom}}$ is $e^{1/R_v}$-strongly log-concave where $R_v=\mathcal{R}_{T[v]}$.
Recall that by \Cref{lem:hom_law_flow} we have the recursive relation
\begin{equation*}
f_{T,\mathrm{Hom}} \propto \prod_{v\in N\pr{v^*}} \pr{f_{T[v],\mathrm{Hom}} * \1_{\cbr{-1,1}}} .
\end{equation*}
By \Cref{prop:strong_log_conc_multiplicativity,prop:conv_strong_log_conc}, $f_{T,\mathrm{Hom}}$ is $\alpha$-strongly log-concave, where
\begin{equation*}
\alpha
= \prod_{v\in N\pr{v^*}} \frac{4}{\pr{e^{-1/R_v}+1}^2} .
\end{equation*}
Applying the bound $e^{-x}\leq \frac{1}{1+x}$ for $x>-1$ twice yields
\begin{equation*}
\frac{2}{e^{-1/R_v}+1}
\geq \frac{2}{\frac{R_v}{R_v+1}+1} =
\frac{1}{1-\frac{1}{2R_v+2}} \geq e^{1/(2R_v+2)} . 
\end{equation*}
Taking a product over $v$ and squaring both sides, we get
\begin{equation*}
\alpha \geq \prod_{v\in N\pr{v^*}} e^{1/\pr{R_v+1}} = e^{1/R},
\end{equation*}
by \Cref{lem:resistance_flow}, completing the induction step.
\end{proof}

Combining \Cref{prop:bounds_on_strong_log_conc_laws} with \Cref{thm:hom_root_law_is_strong_log_conc}, we immediately obtain bounds on the height at the root.

\begin{corollary} \label{cor:upper_bounds_on_hom_root_laws}
Let $(T,v^*)$ be an even finite rooted tree, and let $\phi:V(T)\to\ZZ$ be a uniformly random homomorphism satisfying $\phi\equiv0$ on the leaves of $T$. With $R=\mathcal{R}_T$, the following hold:
\begin{enumerate}[(a)]
    \item $\abs{\phi\pr{v^*}}$ is stochastically dominated by $\abs{\mathcal{N}_{2\ZZ} \pr{0,4R}}$; 
    \item $\var\pr{\phi\pr{v^*}} \leq 4R$;
    \item $\prob{\phi\pr{v^*} = 2n}\leq e^{-n^2/2R}\cdot\prob{\phi\pr{v^*} = 0}$ for all $n\in\ZZ$.
\end{enumerate}
\end{corollary}

\begin{theorem}\label{thm:hom_localization}
Let $T$ be a transient locally finite tree. Then homomorphisms on $T$ are localized.
\end{theorem}

\begin{proof}
Choose any vertex $v^*\in T$ as the root.
For a (finite) odd or even rooted subtree $\pr{T',v^*}\subset \pr{T,v^*}$, let $\phi_{T'}:T'\to\ZZ$ be a uniformly random homomorphism taking the value zero at the leaves of $T'$. We prove that the family of laws of $\phi_{T'}\pr{v^*}$ has a uniformly bounded variance, and is therefore tight.
\Cref{cor:upper_bounds_on_hom_root_laws} implies that when $T'$ is an even tree,
\begin{equation*}
\var\pr{\phi_{T'}\pr{v^*}} \leq 4\mathcal{R}_{T'} \leq 4\mathcal{R}_{T}.
\end{equation*}
If $T'$ is odd, consider the even rooted tree $\pr{T'',v^{**}}$ obtained from $T'$ by adding a new root $v^{**}$ whose only child is $v^*$. For $T''$, \Cref{cor:upper_bounds_on_hom_root_laws} gives
\begin{equation*}
\var\pr{\phi_{T''}\pr{v^{**}}} \leq 4\mathcal{R}_{T''}.
\end{equation*}
However, by \Cref{lem:hom_law_flow} we have $\var\pr{\phi_{T''}\pr{v^{**}}}=\var\pr{\phi_{T'}\pr{v^*}}+1$,
and by \Cref{lem:resistance_flow} we have $\mathcal{R}_{T''}=\mathcal{R}_{T'}+1$.
Combining these, we get
\begin{equation*}
\var\pr{\phi_{T'}\pr{v^*}} \leq 4\pr{\mathcal{R}_{T'}+1}-1
\leq 4\mathcal{R}_T + 3.
\end{equation*}
We thus obtain a uniform upper bound on the variance of $\phi_{T'}\pr{v^*}$ as desired.
\end{proof}

\begin{remark}[localization under general boundary conditions] \label{remark:hom_loc_different_boundary_conds}
The localization result on transient trees extends beyond zero boundary conditions to general strong log-concave boundary conditions. Suppose that the leaf heights follow (a-priori) an $\alpha_0$-strongly log-concave law $f_0:\ZZ\to[0,1]$. Adapting the recursive argument of \Cref{thm:hom_root_law_is_strong_log_conc}, the root law in this case is $\alpha$-strongly log-concave with $\alpha$ determined by the effective resistance in an augmented electrical network (with leaf-adjacent resistances increased by $1/\log\alpha_0$). Because this effective resistance is uniformly bounded across subtrees of a transient tree, the root laws possess uniformly bounded variance. Provided the means are controlled (e.g. if $f_0$ is symmetric), localization follows. This argument also allows the boundary law $f_0$ to vary across the leaves. In particular, when the boundary values are arbitrary fixed values (which may be vary across the leaves), the variance of the root height is uniformly bounded.
\end{remark}

%% file: delocalized.tex
\section{Delocalization of homomorphisms on recurrent trees} \label{sec:delocalization}

In this section we prove the delocalization part of \Cref{thm:equivalence_hom} (\Cref{thm:hom_deloc_on_rec_tree} below) and the lower bounds in \Cref{thm:main_result_hom} (\Cref{cor:lower_bounds_on_hom_root_laws} below), thereby completing the proofs of \Cref{thm:main_result_hom,thm:equivalence_hom}.
Our goal is to acquire lower distribution bounds on $f_{T,\mathrm{Hom}}$ by establishing lower bounds on its log-curvature. The recursive proof structure mirrors the one in the previous section, with the role of strong log-concavity being replaced by the new notion of limited log-concavity.

\subsection{Limited log-concavity}

\begin{definition} \label{defn:hom_limited_log_concave}
Given $\gamma,R>0$, let $\Phi^\gamma_R:(-R,R)\to\RR_+$ be defined by
$\Phi^\gamma_R(x)\coloneqq e^{\gamma/\pr{R-\abs{x}}}$.
We say that $f:Z\to\RRnonn$, where $Z=2\ZZ$ or $Z=2\ZZ+1$, is \emph{$\pr{R,\gamma}$-limited log-concave}, if it is weakly log-concave and
\begin{equation*}
f(n)^2 \leq \Phi^\gamma_R(n)\cdot f(n-2)f(n+2)
\qquad\text{for all } n \in (-R,R)\cap Z.
\end{equation*}
\end{definition}

In contrast to strong log-concavity, which imposes a uniform upper bound on log-curvature, limited log-concavity imposes a lower bound within a restricted interval around zero, which varies across the interval (see \Cref{fig:phi_gamma_r}).

To see why such a restriction is necessary for our purposes, let $T$ be a finite rooted tree. Under the zero boundary condition, the height of the root is distributed within a finite interval $[-d,d]$. Consequently, the log-curvature of $f_{T,\mathrm{Hom}}$ at $x=\pm d$ is zero, so any positive global lower bound on said log-curvature would fail to hold. To help accommodate for this, the lower bound is designed to weaken as $\abs{n}$ increases, smoothly transitioning from a strong constraint near $n=0$ to no constraint at all at $n=\pm R$.

\input{phifig}

Following along the lines of the previous section, we now shift our focus to exploring the behavior of limited log-concavity under convolution with $\1_{\cbr{-1,1}}$ and under products. \Cref{prop:prod_limited_log_conc,prop:conv_limited_log_conc} show that these two operations preserve $\pr{R,\gamma}$-limited log-concavity in a manner consistent with the propagation of the effective resistance up the tree.

\begin{proposition} \label{prop:prod_limited_log_conc}
Let $f_1,f_2:Z\to\RRnonn$ such that $f_i$ is $\pr{R_i,\gamma}$-limited log-concave, where $Z=2\ZZ$ or $Z=2\ZZ+1$. Then $f_1\cdot f_2$ is $\pr{R,\gamma}$-limited log-concave where $\frac{1}{R}=\frac{1}{R_1}+\frac{1}{R_2}$.
\end{proposition}

\begin{proposition} \label{prop:conv_limited_log_conc}
Let $f:Z\to\RRnonn$ be symmetric around 0 and $\pr{R,\gamma}$-limited log-concave with $\gamma\geq8$, where $Z=2\ZZ$ or $Z=2\ZZ+1$. Then $f*\1_{\cbr{-1,1}}$ is $\pr{R+1,\gamma}$-limited log-concave.
\end{proposition}

Before proving these two key properties, we make note of the useful fact that log-concave functions are unimodal.

\begin{proposition} \label{prop:unimodal}
Let $f:Z\to\RRnonn$ be weakly log-concave and symmetric around zero, where $Z=2\ZZ$ or $Z=2\ZZ+1$. 
Then $f(n)$ is decreasing in $|n|$.
\end{proposition}

\begin{proof}
We first consider the case $Z=2\ZZ$. Log-concavity gives $f\pr{0}^2\geq f(-2)f(2) = f\pr{2}^2$ so that $f(0)\geq f(2)$. Thus for all even $k\geq2$ in the support of $f$ we have
\begin{equation*}
\frac{f(k+2)}{f(k)} \leq \frac{f(k)}{f(k-2)} \leq \cdots \leq \frac{f(2)}{f(0)} \leq 1.
\end{equation*}
Since $f$ has contiguous support, $f(n)$ is in fact decreasing for all $n\geq0$, and the rest follows by symmetry. Now consider the case $Z=2\ZZ+1$. For odd $k\geq 3$ in the support of $f$, we have
\begin{equation*}
\frac{f(k+2)}{f(k)} \leq \frac{f(k)}{f(k-2)} \leq \cdots \leq \frac{f(1)}{f(-1)} = 1
\end{equation*}
which is again sufficient by contiguous support and symmetry.
\end{proof}

\begin{proof}[Proof of \Cref{prop:prod_limited_log_conc}]
Denote $g\coloneqq f_1\cdot f_2$. Note that $g$ is weakly log-concave by \Cref{prop:weak_log_conc_closure}.
Let $n\in(-R,R)\cap Z$ be an integer. Since $R\leq\min(R_1,R_2)$, we also have $n\in(-R_1,R_1)$, $n\in(-R_2,R_2)$. Therefore,
\begin{multline*}
g(n)^2 = f_1(n)^2\cdot f_2(n)^2 
\leq \Phi^\gamma_{R_1}(n)\cdot \Phi^\gamma_{R_2}(n)
\cdot f_1 (n-2) f_1 (n+2) \cdot f_2 (n-2) f_2 (n+2) \\
= \Phi^\gamma_{R_1}(n)\cdot \Phi^\gamma_{R_2}(n) \cdot g(n-2) \cdot g(n+2) ,
\end{multline*}
so it suffices to prove that $\Phi^\gamma_{R_1}(n)\cdot\Phi^\gamma_{R_2}(n)\leq\Phi^\gamma_R(n)$. Equivalently,
\begin{equation*}
\frac{\gamma}{R-|n|} -\frac{\gamma}{R_1-|n|} - \frac{\gamma}{R_2-|n|} \geq 0 .
\end{equation*}
Dividing by $\gamma$ and expanding, we get the condition
\begin{equation*}
\frac{R_1 R_2 - R_1R - R_2R + 2R|n| - |n|^2}{(R_1-|n|)(R_2-|n|)(R-|n|)} \geq 0 .
\end{equation*}
Note that since $\frac{1}{R}=\frac{1}{R_1}+\frac{1}{R_2}$, we have $R_1 R_2 = R_1R + R_2R$. Thus the above simplifies to
\begin{equation*}
\frac{2R|n| - |n|^2}{(R_1-|n|)(R_2-|n|)(R-|n|)} = \frac{|n|(2R-|n|)}{(R_1-|n|)(R_2-|n|)(R-|n|)} \geq 0,
\end{equation*}
which holds as $|n|<R,R_1,R_2$.
\end{proof}

\begin{proof}[Proof of \Cref{prop:conv_limited_log_conc}]
Denote $g\coloneqq f*\1_{\cbr{-1,1}}$. Observe that $g$ is weakly log-concave by \Cref{prop:weak_log_conc_closure}. It remains to prove that for all $n\in(-R-1,R+1)\cap\pr{Z+1}$, we have
\begin{equation} \label{ineq:conv_limited_log_conc_pre_expand}
g(n)^2\leq \Phi^\gamma_{R+1}(n) \cdot g(n-2)g(n+2) .
\end{equation}
By symmetry, we may assume that $n \ge 0$. Expanding, \eqref{ineq:conv_limited_log_conc_pre_expand} is equivalent to
\begin{equation} \label{ineq:conv_limited_log_conc_post_expand}
\pr{f(n-1)+f(n+1)}^2\leq \Phi^\gamma_{R+1}(n) \pr{f(n-3)+f(n-1)}\pr{f(n+1)+f(n+3)} .
\end{equation}
We prove \eqref{ineq:conv_limited_log_conc_post_expand} by dividing into three cases based on $(n,R)$. The first case is the main one, while the other two are easy boundary cases. Our strategy in each case is to use limited log-concavity to obtain relations between the values of $f$ at different points, and use them to eliminate the $f(\cdot)$ terms. This leaves us with elementary inequalities on $R,\gamma$ and $n$ which we prove directly.

\paragraph{Case 1: $n+1<R$.} In this case, we have
\begin{equation} \label{ineq:n_plus_one_sub_conc}
f(n+1)^2 \leq \Phi^\gamma_{R}(n+1) f(n-1)f(n+3) .
\end{equation}
Also, since $\abs{n-1}\leq n+1<R$, we have
\begin{equation} \label{ineq:n_minus_one_sub_conc}
f(n-1)^2 \leq \Phi^\gamma_{R}(n-1) f(n-3)f(n+1).
\end{equation}
Multiplying \eqref{ineq:n_plus_one_sub_conc} and \eqref{ineq:n_minus_one_sub_conc} (and safely dividing by $f(n-1)f(n+1)$ here), we get
\begin{equation} \label{ineq:n_pm_one_sub_conc}
f(n-1)f(n+1) \leq \Phi^\gamma_{R}(n-1)\Phi^\gamma_{R}(n+1) f(n-3)f(n+3) .
\end{equation}
Combining \eqref{ineq:n_plus_one_sub_conc}, \eqref{ineq:n_minus_one_sub_conc} and \eqref{ineq:n_pm_one_sub_conc}, we obtain a lower bound on the right-hand side of \eqref{ineq:conv_limited_log_conc_post_expand}:
\begin{multline*}
\Phi^\gamma_{R+1}(n) \cdot \pr{f(n-3)+f(n-1)}\pr{f(n+1)+f(n+3)} = \\
=\Phi^\gamma_{R+1}(n) \cdot \pr{f(n-3)f(n+1) + f(n-1)f(n+1) + f(n-3)f(n+3)+f(n-1)f(n+3)} \\
\geq \Phi^\gamma_{R+1}(n) \cdot
\pr{\frac{f(n-1)^2}{\Phi^\gamma_{R}(n-1)} + f(n-1)f(n+1) +
\frac{f(n-1)f(n+1)}{\Phi^\gamma_{R}(n-1)\Phi^\gamma_{R}(n+1)}+\frac{f(n+1)^2}{\Phi^\gamma_{R}(n+1)}} .
\end{multline*}
Thus, in order to prove \eqref{ineq:conv_limited_log_conc_post_expand}, it suffices to show that
\begin{multline*}
\frac{f(n-1)^2}{\Phi^\gamma_{R}(n-1)} + f(n-1)f(n+1) +
\frac{f(n-1)f(n+1)}{\Phi^\gamma_{R}(n-1)\Phi^\gamma_{R}(n+1)}+\frac{f(n+1)^2}{\Phi^\gamma_{R}(n+1)} \geq \\
\geq
\frac{1}{\Phi^\gamma_{R+1}(n)} \pr{f(n-1)+f(n+1)}^2 ,
\end{multline*}
or equivalently,
\begin{multline}
\pr{\frac{1}{\Phi^\gamma_{R}(n-1)} - \frac{1}{\Phi^\gamma_{R+1}(n)}}f(n-1)^2 +
\pr{\frac{1}{\Phi^\gamma_{R}(n+1)} - \frac{1}{\Phi^\gamma_{R+1}(n)}}f(n+1)^2 + \\
+ \pr{1 + \frac{1}{\Phi^\gamma_{R}(n-1)\Phi^\gamma_{R}(n+1)} - \frac{2}{\Phi^\gamma_{R+1}(n)}}f(n-1)f(n+1) \geq 0 . \label{ineq:before_replacing_f_n_pm_1}
\end{multline}
We aim to substitute both $f(n-1)^2$ and $f(n+1)^2$ above by $f(n-1)f(n+1)$, enabling us to eliminate the terms containing $f$. If $n>0$, we have
\begin{equation}
\label{ineq:replace_f_n_minus_one}
\pr{\frac{1}{\Phi^\gamma_{R}(n-1)} - \frac{1}{\Phi^\gamma_{R+1}(n)}}f(n-1)^2 =
\pr{\frac{1}{\Phi^\gamma_{R}(n-1)} - \frac{1}{\Phi^\gamma_{R+1}(n)}}f(n-1)f(n+1)
\end{equation}
as both sides are zero since $\Phi^\gamma_{R}(n-1) = \Phi^\gamma_{R+1}(n)$, and
\begin{equation}
\label{ineq:replace_f_n_plus_one}
\pr{\frac{1}{\Phi^\gamma_{R}(n+1)} - \frac{1}{\Phi^\gamma_{R+1}(n)}}f(n+1)^2 \geq 
\pr{\frac{1}{\Phi^\gamma_{R}(n+1)} - \frac{1}{\Phi^\gamma_{R+1}(n)}}f(n-1)f(n+1)
\end{equation}
as $f(n+1) \geq f(n-1)$
and $\Phi^\gamma_{R}(n+1) > \Phi^\gamma_{R+1}(n)$.
In the case that $n=0$, \eqref{ineq:replace_f_n_minus_one} and \eqref{ineq:replace_f_n_plus_one} hold with equality as $f(-1)=f(1)$.
After making these substitutions into \eqref{ineq:before_replacing_f_n_pm_1}, it remains to prove that
\begin{multline*}
\pr{\frac{1}{\Phi^\gamma_{R}(n-1)} - \frac{1}{\Phi^\gamma_{R+1}(n)}}f(n-1)f(n+1) +
\pr{\frac{1}{\Phi^\gamma_{R}(n+1)} - \frac{1}{\Phi^\gamma_{R+1}(n)}}f(n-1)f(n+1) + \\
+ \pr{1 + \frac{1}{\Phi^\gamma_{R}(n-1)\Phi^\gamma_{R}(n+1)} - \frac{2}{\Phi^\gamma_{R+1}(n)}}f(n-1)f(n+1) \geq 0 .
\end{multline*}
Dividing by $f(n-1)f(n+1)$ and regrouping, we get the condition
\[ 1 +\frac{1}{\Phi^\gamma_{R}(n-1)}+\frac{1}{\Phi^\gamma_{R}(n+1)}+
\frac{1}{\Phi^\gamma_{R}(n-1)\Phi^\gamma_{R}(n+1)}- \frac{4}{\Phi^\gamma_{R+1}(n)} \geq 0 ,\]
or equivalently,
\[ \pr{1+\frac{1}{\Phi^\gamma_{R}(n-1)}}\pr{1+\frac{1}{\Phi^\gamma_{R}(n+1)}}
\geq \frac{4}{\Phi^\gamma_{R+1}(n)} ,\]
which we prove in \Cref{claim:phi_middle_bound} below.

\paragraph{Case 2: $n+1 \geq R$ and $n>0$.} Note that \eqref{ineq:n_minus_one_sub_conc} still holds, as $\abs{n-1}=n-1<R$ in this case. Applying it together with the trivial inequalities $f(n-1)f(n+3)\geq0$ and $f(n-3)f(n+3)\geq 0$ to the right-hand side of \eqref{ineq:conv_limited_log_conc_post_expand}, the inequality reduces to
\begin{equation*}
f(n-1)^2 + 2f(n-1)f(n+1) + f(n+1)^2 \leq \Phi^\gamma_{R+1}(n) \pr{\frac{f(n-1)^2}{\Phi^\gamma_{R}(n-1)} + f(n-1)f(n+1)} .
\end{equation*}
Since $\Phi^\gamma_{R+1}(n) = \Phi^\gamma_{R}(n-1)$, this is equivalent to
\begin{equation*}
f(n-1)^2 + 2f(n-1)f(n+1) + f(n+1)^2 \leq f(n-1)^2 + \Phi^\gamma_{R+1}(n) f(n-1)f(n+1),
\end{equation*}
or after simplifying,
\begin{equation} \label{ineq:case_2_simplified}
2f(n-1) + f(n+1) \leq \Phi^\gamma_{R+1}(n) f(n-1).
\end{equation}
Observe that $f(n+1) \leq f(n-1)$ holds by unimodality, and that
\begin{equation*}
\Phi^\gamma_{R+1}(n)=e^{\gamma/\pr{R+1-n}}\geq e^{8/2} > 3 .
\end{equation*}
The combination of these inequalities yields \eqref{ineq:case_2_simplified} immediately.

\paragraph{Case 3: $n+1 \geq R$ and $n=0$.} By symmetry around zero, \eqref{ineq:conv_limited_log_conc_post_expand} is equivalent in this case to
\begin{equation*}
\pr{f(1)+f(1)}^2 \leq \Phi^\gamma_{R+1}(0)\pr{f(3)+f(1)}\pr{f(1)+f(3)}
\end{equation*}
or
\begin{equation} \label{ineq:case_3_simplified}
4f(1)^2 \leq \Phi^\gamma_{R+1}(0)\pr{f(3)+f(1)}^2 .
\end{equation}
Observe that
\begin{equation*}
\Phi^\gamma_{R+1}(0) = e^{\gamma/\pr{R+1}}\geq e^{8/2} > 4 ,
\end{equation*}
which immediately yields \eqref{ineq:case_3_simplified}.

Having addressed all three cases, the argument is complete modulo \Cref{claim:phi_middle_bound} which we state and prove below.
\end{proof}

\begin{claim} \label{claim:phi_middle_bound}
Let $n\geq0$ be an integer, and let $R>n+1$, $\gamma\geq 8$. Then
\begin{equation} \label{ineq:phi_middle_bound_lemma}
\pr{1+\frac{1}{\Phi^\gamma_{R}(n-1)}}\pr{1+\frac{1}{\Phi^\gamma_{R}(n+1)}}
\geq \frac{4}{\Phi^\gamma_{R+1}(n)}.
\end{equation}
\end{claim}

\begin{proof}
We begin with the case $n=0$, which translates to
\begin{equation*}
\pr{1+\frac{1}{\Phi^\gamma_{R}(-1)}}\pr{1+\frac{1}{\Phi^\gamma_{R}(1)}}
\geq \frac{4}{\Phi^\gamma_{R+1}(0)},
\end{equation*}
or after taking a square root,
\begin{equation*}
1+\exp\pr{\frac{-\gamma}{R-1}} \geq 2\exp\pr{\frac{-\gamma}{2(R+1)}}.
\end{equation*}
Denote $F(x)\coloneqq 1+\exp\pr{\frac{-\gamma}{x-1}}-2\exp\pr{\frac{-\gamma}{2(x+1)}}$. We need to show that $F(R)\geq0$ for all $R>1$. Observe that $\lim_{R\to\infty} F(R)=0$, so it suffices to prove that $F$ is decreasing. Differentiating, we need to show that
\begin{equation*}
F'(R) = \frac{\gamma}{\pr{R-1}^2} \exp\pr{\frac{-\gamma}{R-1}} -
\frac{\gamma}{\pr{R+1}^2} \exp\pr{\frac{-\gamma}{2(R+1)}} \leq 0 .
\end{equation*}
After rearranging and taking a square root, we get the equivalent condition
\begin{equation} \label{ineq:equiv_to_f_decreasing_in_claim}
\exp\pr{\frac{\gamma}{2(R-1)} - \frac{\gamma}{4(R+1)}}
\geq \frac{R+1}{R-1}.
\end{equation}
To verify \eqref{ineq:equiv_to_f_decreasing_in_claim} we apply $e^x \geq 1+x$ and find that
\begin{multline*}
\exp\pr{\frac{\gamma}{2(R-1)} - \frac{\gamma}{4(R+1)}} \geq 
1 + \frac{\gamma}{2(R-1)} - \frac{\gamma}{4(R+1)} \geq \\
\geq 1 + \frac{\gamma}{2(R-1)} - \frac{\gamma}{4(R-1)} 
\geq 1 + \frac{2}{R-1} = \frac{R+1}{R-1}.
\end{multline*}
To finish the proof of the claim, we consider the case $n>0$. Here, \eqref{ineq:phi_middle_bound_lemma} translates to
\begin{equation*}
\pr{1+\exp\pr{\frac{-\gamma}{R-n+1}}}\pr{1+\exp\pr{\frac{-\gamma}{R-n-1}}}
\geq 4\exp\pr{\frac{-\gamma}{R-n+1}}.
\end{equation*}
We may replace $\exp\pr{\frac{-\gamma}{R-n+1}}$ by $\exp\pr{\frac{-\gamma}{R-n-1}}$ and prove the stronger inequality
\begin{equation*}
\pr{1+\exp\pr{\frac{-\gamma}{R-n-1}}}^2
\geq 4\exp\pr{\frac{-\gamma}{R-n+1}},
\end{equation*}
which is equivalent to the already proven $F\pr{R-n}\geq0$ .
\end{proof}

As in \Cref{sec:localization}, the conservation properties of limited log-concavity under convolution with $\1_{\cbr{-1,1}}$ and under products provide the core of the induction step for establishing limited log-concavity of $f_{T,\mathrm{Hom}}$.
We proceed to develop lower distribution bounds used to leverage this property.

\begin{proposition} \label{prop:bounds_on_limited_log_conc_laws}
Let $X$ be a $2\ZZ$-valued random variable whose law is symmetric around zero and $\pr{R,\gamma}$-limited log-concave with $\gamma, R>0$. Then the following hold:
\begin{enumerate}[(a)]
    \item $\abs{X}$ stochastically dominates $\abs{\mathcal{N}_{2\ZZ\cap(-R/2,R/2)}\pr{0,\frac{2R}{\gamma}}}$; \label{item:limited_stoch_dom}
    \item $\var{X} \geq \frac{2R}{\gamma e^4 \sqrt\pi}$ whenever $R \geq \max\pr{4\gamma, 64/\gamma}$; \label{item:limited_variance_bound}
    \item $\prob{X = 2n}\geq e^{-\gamma n^2/R}\cdot \prob{X = 0}$ for all integers $-R/4 < n < R/4$; \label{item:limited_pointwise_bound}
    \item $\prob{X = 2n}\leq \frac{\max\pr{e^\gamma,4}}{\sqrt{R}}$ for all $n\in\ZZ$. \label{item:limited_pointwise_upper_bound}
\end{enumerate}
\end{proposition}

\begin{proof}
Let $f:2\ZZ\to[0,1]$ be the law of $X$.
We first show~\ref{item:limited_stoch_dom}.
As in the proof of \Cref{prop:bounds_on_strong_log_conc_laws}, we wish to bound the successive quotients $f(2n+2)/f(2n)$ to employ \Cref{lem:likelyhood_dom_implies_stoch_dom}. Since $f$ is symmetric and unimodal by \Cref{prop:unimodal}, its maximum is at zero, and therefore $f(0) > 0$. For all integers $n\in(-R/4,R/4)$, limited log-concavity gives
\begin{equation} \label{ineq:limited_log_conc_applied_for_stoch_dom}
f(2n)^2 \leq \Phi^\gamma_R(2n) \cdot f(2n-2)f(2n+2) \leq e^{2\gamma/R} \cdot f(2n-2)f(2n+2).
\end{equation}
In particular, for $n=0$ we deduce by symmetry that
\begin{equation*}
f(0)^2 \leq e^{2\gamma/R} \cdot f(-2)f(2) = e^{2\gamma/R} \cdot f(2)^2
\end{equation*}
so that $f(0) \leq e^{\gamma/R}f(2)$. Because $f(0) > 0$, this inequality forces $f(2) > 0$. By induction, it follows that $f(2n) > 0$ for all integers $n \in (-R/4, R/4)$. Thus, we can safely divide \eqref{ineq:limited_log_conc_applied_for_stoch_dom} by values of $f$ within this interval. For all integers $n\in[0,R/4)$, we iteratively deduce that
\begin{equation}
\label{ineq:delocalization_stoch_dom_helper_ineqs}
\frac{f(2n+2)}{f(2n)} \geq e^{-(2n+1)\gamma/R}.
\end{equation}
We wish to prove that $\abs{X}$ stochastically dominates $\abs{Y}$, where $Y=\mathcal{N}_{2\ZZ\cap(-R/2,R/2)}\pr{0,\frac{2R}{\gamma}}$ is a truncated Gaussian.
The laws of $\abs{X}$ and $\abs{Y}$ are respectively given by
\begin{equation*}
f_{\abs{X}}(2n)=\begin{cases}
			f(0), & n=0\\
            2f(2n), & n > 0
		 \end{cases}
\end{equation*}
and
\begin{equation*}
f_{\abs{Y}}(2n)\propto\begin{cases}
			1, & n=0\\
            2e^{-\gamma n^2/R}, & 0 < n < R/4 \\
            0, & n \geq R/4
		 \end{cases}.
\end{equation*}
For $0 < n < R/4-1$, \eqref{ineq:delocalization_stoch_dom_helper_ineqs} gives
\begin{equation*}
\frac{f_{\abs{X}}\pr{2n+2}}{f_{\abs{X}}\pr{2n}} \geq e^{-(2n+1)\gamma/R} = \frac{f_{\abs{Y}}\pr{2n+2}}{f_{\abs{Y}}\pr{2n}},
\end{equation*}
while for $n=0$, we have
\begin{equation*}
\frac{f_{\abs{X}}\pr{2}}{f_{\abs{X}}\pr{0}} \geq 2e^{-\gamma/R} = \frac{f_{\abs{Y}}\pr{2}}{f_{\abs{Y}}\pr{0}}.
\end{equation*}
Since $2n$ is outside the support of $f_{\abs{Y}}$ whenever $n\geq R/4$,
\Cref{lem:likelyhood_dom_implies_stoch_dom} implies that $\abs{Y}$ is stochastically dominated by $\abs{X}$, which establishes \ref{item:limited_stoch_dom}. This stochastic dominance also yields
$\var{X} \geq \var{Y}$.
Thus to prove \ref{item:limited_variance_bound}, we provide a lower bound on the variance of the truncated Gaussian $Y$. A direct calculation gives
\begin{equation*}
\var{Y} = \frac{\sum_{n=1}^{\floor{R/4}} 8n^2 e^{-\gamma n^2/R}}{1 + \sum_{n=1}^{\floor{R/4}} 2e^{-\gamma n^2/R}}.
\end{equation*}
To lower-bound the numerator, we restrict the sum to the interval $I = \br{\sqrt{R/\gamma}, 2\sqrt{R/\gamma}}$.
The assumption $R \geq \max\pr{4\gamma, 64/\gamma}$ ensures that $I$ is contained within the limits of the original sum, and contains at least $\frac{1}{2}\sqrt{R/\gamma}$ integers.
For all $n \in I$, we have $n^2 \geq R/\gamma$ and $e^{-\gamma n^2/R} \geq e^{-4}$. The numerator is thus bounded below by
\begin{equation*}
8 \pr{\frac{1}{2}\sqrt{\frac{R}{\gamma}}} \pr{\frac{R}{\gamma}} e^{-4} = \frac{4R^{3/2}}{e^4 \gamma^{3/2}}.
\end{equation*}
For the denominator, we use the standard integral upper bound:
\begin{equation*}
1 + 2\sum_{n=1}^{\infty} e^{-\gamma n^2/R} \leq 1 + 2\int_0^\infty e^{-\gamma x^2/R} dx = 1 + \sqrt{\frac{\pi R}{\gamma}}
\leq 2 \sqrt{\frac{\pi R}{\gamma}}.
\end{equation*}
Combining these bounds, we obtain the desired lower bound for the variance of $X$:
\begin{equation*}
\var{X} \geq \var{Y} \geq \frac{2 R}{\gamma e^4 \sqrt{\pi}}.
\end{equation*}
To prove \ref{item:limited_pointwise_bound}, observe that for $0<n<R/4$ we have by \eqref{ineq:delocalization_stoch_dom_helper_ineqs} that
\begin{equation*}
\frac{f(2n)}{f(0)} = \frac{f(2n)}{f(2n-2)} \cdot \cdots \cdot \frac{f(2)}{f(0)}
\geq e^{-(2n-1)\gamma/R} \cdot \cdots \cdot e^{-\gamma/R} = e^{-\gamma n^2/R},
\end{equation*}
and by symmetry the same is true for $-R/4<n<0$.

Finally, for \ref{item:limited_pointwise_upper_bound} we may assume that $R>16$, otherwise the inequality holds trivially. In particular, $\sqrt{R}<R/4$. Therefore, for all integers $0\leq n\leq \sqrt{R}$, \ref{item:limited_pointwise_bound} gives
\begin{equation*}
f(2n)\geq  e^{-\gamma n^2/R} \cdot f(0) \geq e^{-\gamma} \cdot f(0).
\end{equation*}
Summing over $n$, we get
\begin{equation*}
\sqrt{R} \cdot e^{-\gamma} \cdot f(0) \leq \sum_{n=0}^{\floor{\sqrt{R}}} f(2n) \leq 1
\end{equation*}
so that $f(0) \leq \frac{e^\gamma}{\sqrt{R}}$.
Finally, since $f$ is symmetric around zero and unimodal by \Cref{prop:unimodal}, this upper bound holds for all values of $f$.
\end{proof}

\subsection{Proof of delocalization}\label{sec:hom-delocalization}
\begin{theorem} \label{thm:hom_root_law_is_limited_log_conc}
Let $(T,v^*)$ be a odd or even finite rooted tree. Then $f_{T,\mathrm{Hom}}$ is $\pr{R,\gamma}$-limited log-concave for 
$R=\mathcal{R}_T$ and $\gamma=8$.
\end{theorem}

\begin{proof}
We induct on the depth of the tree. For the trivial tree of depth zero, the law $f_{T,\mathrm{Hom}}=\1_{\cbr{0}}$ is $(0,\gamma)$-limited log-concave (this degenerates to the sole requirement that $f_{T,\mathrm{Hom}}$ is weakly log-concave).

Suppose that $T$ is a nontrivial tree. Applying the induction hypothesis to $T[v]$ for $v\in N\pr{v^*}$, we obtain that $f_{T[v],\mathrm{Hom}}$ is $\Phi^{8}_{R_v}$-limited log-concave where $R_v=\mathcal{R}_{T[v]}$.
Recall that \Cref{lem:hom_law_flow} gives the recursive formula
\begin{equation*}
f_{T,\mathrm{Hom}} \propto \prod_{v\in N\pr{v^*}} \pr{f_{T[v],\mathrm{Hom}} * \1_{\cbr{-1,1}}} .
\end{equation*}
By \Cref{prop:prod_limited_log_conc,prop:conv_limited_log_conc} (noting that $f_{T,\mathrm{Hom}}$ is symmetric around 0) and \Cref{lem:resistance_flow}, the law $f_{T,\mathrm{Hom}}$ is $\Phi^{8}_{R}$-limited log-concave where
\begin{align*}
R
= \sum_{v\in N\pr{v^*}} \frac{1}{R_v+1}
= \mathcal{R}_{T} ,
\end{align*}
completing the induction step.
\end{proof}

\begin{corollary} \label{cor:lower_bounds_on_hom_root_laws}
Let $(T,v^*)$ be an even finite rooted tree, and let $\phi:V(T)\to\ZZ$ be a uniformly random homomorphism satisfying $\phi\equiv0$ on the leaves of $T$. With $R=\mathcal{R}_T$, the following hold:
\begin{enumerate}[(a)]
    \item $\abs{\phi\pr{v^*}}$ stochastically dominates $\abs{\mathcal{N}_{2\ZZ\cap(-R/2,R/2)}\pr{0,\frac{R}{4}}}$;
    \item $\var\pr{\phi\pr{v^*}} \geq \floor{\frac{R}{4e^4\sqrt{\pi}}}$;
    \item $\prob{\phi\pr{v^*} = 2n}\geq e^{-8n^2/R}\cdot \prob{\phi\pr{v^*} = 0}$ for all integers $-R/4 < n < R/4$;
    \item $\prob{\phi\pr{v^*} = 2n}\leq \frac{e^8}{\sqrt{R}}$ for all $n$.
\end{enumerate}
\end{corollary}

\begin{theorem} \label{thm:hom_deloc_on_rec_tree}
Let $T$ be a recurrent locally finite tree. Then homomorphisms on $T$ are delocalized.
\end{theorem}

\begin{proof}
Choose any vertex $v^*\in T$ as the root.
Let $T_1 \subset T_2 \subset T_3 \subset ... \subset T$ be an increasing sequence of odd or even finite subtrees of $T$ rooted in $v^*$, satisfying $\bigcup_{i=1}^{\infty} T_i = T$. Since $T$ is recurrent, we have $\mathcal{R}_{T_i} \to \infty$ as $i\to\infty$.

For every $i=1,2,3,...$, let $\phi_{T_i}:T_i\to\ZZ$ be a uniformly random homomorphism taking the value zero at the leaves of $T_i$. We show that $\abs{\phi_{T_i}\pr{v^*}}\to\infty$ in probability as $i \to \infty$, which establishes delocalization.
When $T_i$ is an even tree, \Cref{cor:lower_bounds_on_hom_root_laws} implies that for any fixed $A>0$, we have
\begin{equation*}
\prob{\abs{\phi_{T_i}\pr{v^*}} < A} \leq \frac{e^8 (2A - 1)}{\sqrt{\mathcal{R}_{T_i}}} \overset{i\to\infty}{\longrightarrow} 0.
\end{equation*}
If $T_i$ is an odd tree, consider the even tree $\pr{T'_i,v^{**}}$, obtained from $T_i$ by adding a new root $v^{**}$ whose only child is $v^*$. Since $\mathcal{R}_{T'_i} = \mathcal{R}_{T_i}+1 \to \infty$, we get that $|\phi_{T'_i}\pr{v^{**}}| \to \infty$ in probability by what we have just shown. Since $|\phi_{T_i}\pr{v^*}| \ge_{st} |\phi_{T'_i}\pr{v^{**}}| - 1$, we see that $|\phi_{T_i}\pr{v^*}| \to \infty$ in probability as well, which completes the proof of delocalization.
\end{proof}

\begin{remark}[delocalization under general boundary conditions] \label{remark:hom_deloc_different_boundary_conds}
Delocalization on recurrent trees readily extends to general symmetric, log-concave boundary conditions. 
Assuming that the leaf heights a-priori follow a symmetric and weakly log-concave law $f_0:\ZZ\to[0,1]$, the recursive proof in \Cref{thm:hom_root_law_is_limited_log_conc} is entirely unchanged.
This also holds when allowing $f_0$ to vary across the leaves.
Note that one cannot entirely drop the symmetry condition since it is possible, for example, to choose deterministic boundary conditions which force the height at the root to be zero.
\end{remark}

%% file: phifig.tex
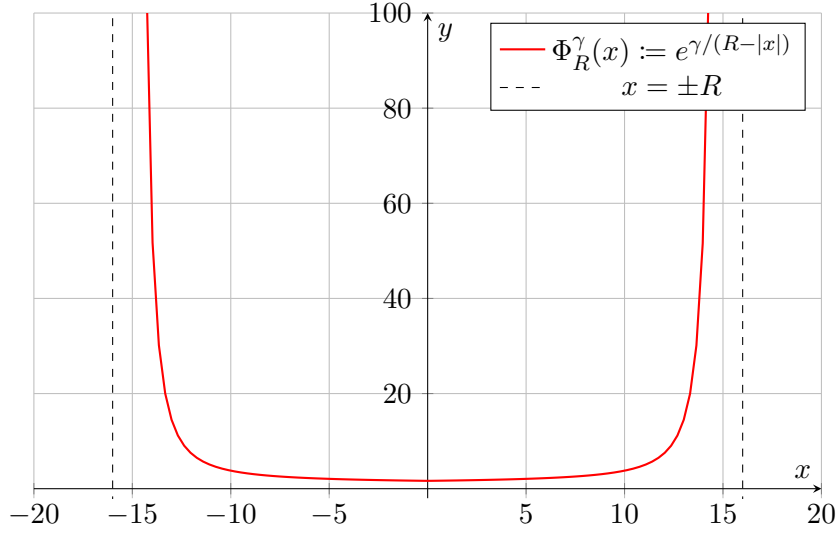
\begin{figure}[t] 
    \centering 
    
    \begin{tikzpicture}
    \begin{axis}[
        axis lines = center,
        xlabel = {$x$},
        ylabel = {$y$},
        xmin = -20, xmax = 20,
        ymin = -2, ymax = 100,
        domain = -15.9:15.9,     
        samples = 100,       
        restrict y to domain=-2:200,
        grid = major,
        width = 12cm,
        height = 8cm
    ]
    
    \addplot [
        color=red, 
        thick
    ] {exp(8/(16-abs(x))};
    \addlegendentry{$\Phi^{\gamma}_R(x)\coloneqq e^{\gamma/\pr{R-\abs{x}}}$}

    \draw [dashed] (-16,-10) -- (-16,200);
    \draw [dashed] (16,-10) -- (16,200);
    \addlegendimage{dashed}
    \addlegendentry{$x=\pm R$}
    \end{axis}
    \end{tikzpicture}
    \caption{The graph of $y=\Phi^{\gamma}_R(x)$ for $\gamma=8$,  $R=16$. This function is defined on $(-R,R)$. Smaller values near the center of the interval translate to stronger constraints on the log-curvature of $(R,\gamma)$-limited log-concave functions at those points.}
    \label{fig:phi_gamma_r}
\end{figure}

%% file: lipschitz.tex
\section{Lipschitz functions} \label{sec:lipschitz}
In this section we prove \Cref{thm:equivalence_lip} (split into \Cref{thm:lip_loc_on_tr_tree,thm:lip_deloc_on_rec_tree} below) by adapting the recursive argument to support Lipschitz functions. The change stems from using a different convolution kernel, $\1_{\cbr{-M,...,M}}$, instead of $\1_{\cbr{-1,1}}$. This warrants amending the notions of strong and limited log-concavity used, so that they are maintained by convolution with this kernel.

A more technical change is that the distributions discussed in this section are no longer even-valued or odd-valued, as Lipschitz functions do not have the same parity restrictions as homomorphisms. This naturally causes a slight change to the definitions used in this section.

Given a finite rooted tree $\pr{T,v^*}$, we denote by $f_{T, M-\mathrm{Lip}}$ the law of $\phi\pr{v^*}$, where $\phi:T\to\ZZ$ is a uniformly random $M$-Lipschitz function taking the value zero at the leaves.

\subsection{Localization on transient trees}
The proof for localization of random Lipschitz functions follows along the same lines as the argument for random homomorphisms in \Cref{sec:localization}: Recursively establishing strong log-concavity of the law at the root, and producing distribution bounds accordingly. We begin by giving a refined version of strong log-concavity.

\begin{definition}
We say that $f:\ZZ\to\RRnonn$ is \emph{(weakly) log-concave} if its support consists of a contiguous sequence of integers, and
\begin{equation*}
f(n)^2 \geq f(n-1)f(n+1) \qquad\text{for all } n \in \ZZ.
\end{equation*}
\end{definition}

\begin{definition} \label{defn:alpha_m_strong_log_conc}
Let $\alpha>1$ and let $m$ be a positive integer. We say that $f:\ZZ\to\RRnonn$ is \emph{$\pr{\alpha,m}$-strongly log-concave} if it is weakly log-concave and
\begin{equation*}
f(n+1)f(n+m) \geq \alpha\cdot f(n)f(n+m+1)
\qquad\text{for all } n \in \ZZ.
\end{equation*}
\end{definition}

As opposed to $\alpha$-strong log-concavity, which imposes a uniform upper bound of $\frac{1}{\alpha}$ on the log-curvature, $(\alpha,m)$-strong log-concavity only requires a uniform upper bound of $1$, combined with a stricter upper bound on the average log-curvature over intervals of length $m$.

Crucially, this new notion of strong log-concavity is conserved under convolution with the indicator function of an interval of length $m$, as well as under products.

\begin{proposition} \label{prop:product_lipschitz_log_conc}
Let $f_1,f_2:\ZZ\to\RRnonn$ such that $f_i$ is $\pr{\alpha_i,m}$-strongly log-concave. Then $f_1\cdot f_2$ is $\pr{\alpha_1\alpha_2, m}$-strongly log-concave.
\end{proposition}

\begin{proposition} \label{prop:convolve_with_offset_m_lip_kernel}
Let $f:\ZZ\to\RRnonn$ be $(\alpha,m)$-strongly log-concave. Then $f*\1_{\cbr{0,1,...,m-1}}$ is $(\beta,m)$-strongly log-concave for $\beta=\frac{m}{m+\alpha^{-m}-1}$.
\end{proposition}

$(\alpha,m)$-strong log-concavity also gives distribution bounds that are analogous to \Cref{prop:bounds_on_strong_log_conc_laws}.

\begin{proposition} \label{prop:bounds_on_strong_log_conc_laws_lipschitz}
Let $X$ be an integer-valued random variable, whose law is symmetric around zero and $(\alpha,2M+1)$-strongly log-concave with $\alpha>1$. Then the following hold:
\begin{enumerate}[(a)]
    \item $\abs{X}$ is stochastically dominated by
    $M+1+\abs{\mathcal{N}_{\ZZ}\pr{0,\frac{4M+2}{\log\alpha}}}$; \label{item:lip_strong_stoch_dom}
    \item $\var{X} \leq 2(M+1)^2 + \frac{8M+4}{\log\alpha}$; \label{item:lip_strong_variance_bound}
    \item $\prob{X = n}\leq \alpha^{-\frac{(n-M-1)^2 - 1}{8M+4}}\cdot \prob{X = 0}$ for all integers $n>M$. \label{item:lip_strong_pointwise_bound}
\end{enumerate}
\end{proposition}

Before proving \Cref{prop:product_lipschitz_log_conc,prop:convolve_with_offset_m_lip_kernel,prop:bounds_on_strong_log_conc_laws_lipschitz}, we briefly describe the remainder of the argument proving localization of Lipschitz functions.

\begin{corollary} \label{cor:convolve_with_m_lip_kernel}
Let $f:\ZZ\to\RRnonn$ be $(\alpha,2M+1)$-strongly log-concave. Then $f*\1_{\cbr{-M,...,M}}$ is $(\beta,2M+1)$-strongly log-concave for $\beta=\frac{2M+1}{2M+\alpha^{-2M-1}}$.
\end{corollary}

\begin{proof}
Follows directly from \Cref{prop:convolve_with_offset_m_lip_kernel} and translation invariance of strong log-concavity.
\end{proof}

\begin{theorem} \label{thm:lip_root_law_is_strong_log_conc}
Let $M\geq1$ and let $(T,v^*)$ be a finite rooted tree. Then the law $f_{T,M-\mathrm{Lip}}$ is $\pr{e^{1/(2M+1)R},2M+1}$-strongly log-concave for $R=\mathcal{R}_T$.
\end{theorem}

\begin{proof}
We induct on the depth of the tree. For the trivial tree $T$ of height zero, $f_{T,M-\mathrm{Lip}}=\1_{\cbr{0}}$ is $(\infty,2M+1)$-strongly log-concave (that is, $(\alpha,2M+1)$-strongly log-concave for all $\alpha>1$).

Suppose that $T$ is nontrivial.
A counting argument along the lines of the proof of \Cref{lem:hom_law_flow} shows that
\begin{equation*}
f_{T,M-\mathrm{Lip}} \propto \prod_{v\in N\pr{v^*}} \pr{f_{T[v],M-\mathrm{Lip}} * \1_{\cbr{-M,...,M}}} .
\end{equation*}
By the induction hypothesis, each $f_{T[v],M-\mathrm{Lip}}$ is $\pr{e^{1/(2M+1)R_v},2M+1}$-strongly log-concave where $R_v=\mathcal{R}_{T[v]}$.
By \Cref{prop:product_lipschitz_log_conc} and \Cref{cor:convolve_with_m_lip_kernel}, $f_{T,M-\mathrm{Lip}}$ is $\pr{\alpha,2M+1}$-strongly log-concave where
\begin{align*}
\alpha
= \prod_{v\in N\pr{v^*}} \frac{2M+1}{2M+e^{-1/R_v}} .
\end{align*}
Applying the bound $e^{-x}\leq \frac{1}{1+x}$ for $x>-1$ twice, we get
\begin{equation*}
\frac{2M+1}{2M+e^{-1/R_v}} \geq \frac{2M+1}{2M+\frac{R_v}{R_v+1}} =
\frac{1}{1-\frac{1}{(2M+1)\pr{R_v+1}}} \geq e^{1/(2M+1)\pr{R_v+1}} .
\end{equation*}
Multiplying this over $v$, we have
\begin{align*}
\alpha
\geq \prod_{v\in N\pr{v^*}}e^{1/(2M+1)\pr{R_v+1}}
= e^{1/(2M+1)R}
\end{align*}
by \Cref{lem:resistance_flow}, completing the induction step.
\end{proof}

\begin{corollary} \label{cor:upper_bounds_on_lip_root_laws}
Let $(T,v^*)$ be an even finite rooted tree, and let $M$ be a positive integer. Let $\phi:V(T)\to\ZZ$ be a uniformly random $M$-Lipschitz function satisfying $\phi\equiv0$ on the leaves of $T$. With $R=\mathcal{R}_T$, the following hold:
\begin{enumerate}[(a)]
    \item $\abs{\phi\pr{v^*}}$ is stochastically dominated by
    $M+1+\abs{\mathcal{N}_{\ZZ}\pr{0,\pr{2M+1}^2 R}}$; 
    \item $\var\pr{\phi\pr{v^*}} \leq \pr{4M+2}^2 (R+1)$.
\end{enumerate}
\end{corollary}

\begin{theorem} \label{thm:lip_loc_on_tr_tree}
Let $T$ be a transient locally finite tree. Then integer-valued $M$-Lipschitz functions on $T$ are localized for all integers $M\geq 1$.
\end{theorem}

\begin{proof}
Choose any vertex $v^*\in T$ as the root.
Let $\pr{T',v^*}\subset \pr{T,v^*}$ be a finite rooted subtree, and let $\phi_{T'}:T'\to\ZZ$ be a uniformly random $M$-Lipschitz function taking the value zero at the leaves of $T'$. By \Cref{cor:upper_bounds_on_lip_root_laws},
\begin{equation*}
\var\pr{\phi_{T'}\pr{v^*}}
\leq \pr{4 M+2}^2 \pr{\mathcal{R}_{T'} + 1}
\leq \pr{4 M+2}^2 \pr{\mathcal{R}_{T} + 1} .
\end{equation*}Thus the family of laws of $\phi_{T'}\pr{v^*}$ has a uniformly bounded variance, and is therefore tight.
\end{proof}

As another corollary, we produce a localization result for the real-valued Lipschitz model by essentially letting $M\to\infty$, along the same lines as in \cite{PSY13lip}.
\begin{corollary}
Let $(T,v^*)$ be a transient locally finite rooted tree. Then real-valued 1-Lipschitz functions on $T$ are localized.
\end{corollary}

\begin{proof}
Let $\pr{T',v^*}\subset \pr{T,v^*}$ be a finite subtree. Let $\phi:T'\to\RR$ be a uniformly random 1-Lipschitz function, and let $\mu$ be its distribution. We may think of $\mu$ as the uniform distribution on the polytope $\mathcal{L}_{T'}\subset\RR^{T'}$ of all such Lipschitz functions.
Likewise, for $M\geq1$, let $\phi_M:T'\to\ZZ$ be a uniformly random $M$-Lipschitz function, and let $\mu_M$ be the distribution of $\phi_M/M$. Observe that $\mu_M$ is the uniform distribution on the discrete set $\mathcal{L}_{T'}\cap \pr{\frac{1}{M}\ZZ}^{T'}$. Since $\mathcal{L}_{T'}$ is a bounded polytope,
the measures $\mu_M$ weakly converge to $\mu$. In particular, by integrating the bounded function $\mathcal{L}_{T'}\to\RR$ given by $\phi\mapsto \phi\pr{v^*}^2$ with respect to these measures, we get
\begin{equation*}
\var\pr{\phi\pr{v^*}} = \int{\phi\pr{v^*}^2 d\mu}
= \lim_{M\to\infty} \frac{1}{M^2}\int{\phi_M\pr{v^*}^2 d\mu_M}
= \lim_{M\to\infty} \frac{1}{M^2} \var\pr{\phi_M\pr{v^*}}
\leq 16(R+1)
\end{equation*}
by \Cref{cor:upper_bounds_on_lip_root_laws}. Therefore the variance of the height at the root is uniformly bounded as the subtree $T'$ varies, which yields localization.
\end{proof}

To conclude this subsection, we prove \Cref{prop:product_lipschitz_log_conc,prop:convolve_with_offset_m_lip_kernel,prop:bounds_on_strong_log_conc_laws_lipschitz}.

\begin{proof}[Proof of \Cref{prop:product_lipschitz_log_conc}]
Denote $g\coloneqq f_1\cdot f_2$. For any $n\in\ZZ$, we have
\begin{multline*}
g(n+1)g(n+m)=f_1(n+1)f_1(n+m)\cdot f_2(n+1)f_2(n+m) \geq \\
\geq \alpha_1 \alpha_2
\cdot f_1 (n) f_1 (n+m+1) \cdot f_2 (n) f_2 (n+m+1)
= \alpha_1 \alpha_2 \cdot g(n) g(n+m+1) . \qedhere
\end{multline*}
\end{proof}

Before proving \Cref{prop:convolve_with_offset_m_lip_kernel}, we develop a few useful tools describing the behavior of log-concave functions over larger intervals.

\begin{claim} \label{claim:weak_concave_over_distance}
Let $f:\ZZ\to\br{0,1}$ be log-concave. Then whenever $i\geq j$, we have
\begin{equation*}
f(i+1)f(j-1) \leq f(i)f(j) .
\end{equation*}
\end{claim} 

\begin{proof}

If $f(i+1)f(j-1)$ then the inequality holds trivially. Thus we may assume that $f(i+1)$ and $f(j-1)$ are positive, which means that $f(k)>0$ whenever $j-1\leq k \leq i+1$. Repeatedly applying log-concavity in the form $\frac{f(x)}{f(x-1)} \geq \frac{f(x+1)}{f(x)}$, we get
\begin{equation*}
\frac{f(j)}{f(j-1)} \geq \frac{f(j+1)}{f(j)} \geq \cdots \geq \frac{f(i+1)}{f(i)} . \qedhere
\end{equation*}
\end{proof}

\begin{claim} \label{claim:weak_concave_over_distance_sym}
Let $f:\ZZ\to\br{0,1}$ be log-concave. Then whenever $i\leq j\leq r/2$, we have
\begin{equation*}
f(i)f(r-i)\leq f(j)f(r-j) .
\end{equation*}
\end{claim}

\begin{proof}
By \Cref{claim:weak_concave_over_distance}, for every $i\leq k \leq j$ we have
\begin{equation*}
f(k-1)f(r-k+1)\leq f(k)f(r-k)
\end{equation*}
since $r-k\geq k$. Repeatedly applying this, we get
\begin{equation*}
f(i)f(r-i)\leq f(i+1)f(r-i-1)\leq ... \leq f(j)f(r-j) . \qedhere
\end{equation*}
\end{proof}

\begin{claim} \label{claim:strong_concave_over_distance}
Let $f:\ZZ\to\br{0,1}$ be $\pr{\alpha, m}$-strongly log-concave. Then whenever $x\leq y$, we have
\begin{equation*}
f(y)f(x+m)\geq \alpha^{y-x} f(x)f(y+m) .
\end{equation*}
\end{claim}

\begin{proof}
If $f(x)f(y+m) = 0$, the inequality holds trivially. Thus we may assume that $f(x)$ and $f(y+m)$ are positive. By the contiguous support of $f$, we are guaranteed that $f(k) > 0$ for all integers $k \in [x, y+m]$. Repeatedly applying $(\alpha,m)$-strong log-concavity, we get,
\begin{equation*}
\frac{f(x+m)}{f(x)} \geq \frac{\alpha f(x+m+1)}{f(x+1)}
\geq \cdots \geq \frac{\alpha^{y-x} f(y+m)}{f(y)} . \qedhere
\end{equation*}
\end{proof}

In the proof of \Cref{prop:convolve_with_offset_m_lip_kernel}, we use closure of weak log-concavity under convolution. This is one of the well-known properties of log-concavity, see for example \cite[Theorem~4.1(a)]{SW14}.
We include a short and self-contained proof for the reader's convenience.

\begin{proposition} \label{prop:discrete_weak_concave_after_convolve}
Let $p,q:\ZZ\to\RRnonn$ be log-concave. Then $p*q$ is log-concave.
\end{proposition}

\begin{proof}
Fix $x\in\ZZ$. For every $i,j\in\ZZ$, we have
\begin{equation*}
\pr{p(i)p(j+1)-p(i+1)p(j)}\cdot\pr{q(x-i)q(x-j-1)-q(x-i-1)q(x-j)} \geq 0.
\end{equation*}
To see this, consider three cases.
If $i\geq j+1$, then both terms in parentheses above are nonnegative by applying \Cref{claim:weak_concave_over_distance} on $p,q$ at indices $(i,j+1)$ and $(x-j-1,x-i)$, respectively.
Symmetrically, if $j \geq i+1$, then both terms are nonpositive.
Finally, if $i=j$, then both terms are zero.

Now sum over all pairs $(i,j)$ to get
\begin{align*}
0\leq &\sum_{i,j\in\ZZ}\pr{p(i)p(j+1)-p(i+1)p(j)}\cdot\pr{q(x-i)q(x-j-1)-q(x-i-1)q(x-j)} =\\
= &\sum_{i,j\in\ZZ} p(i)p(j+1)q(x-i)q(x-j-1) + \sum_{i,j\in\ZZ} p(i+1)p(j)q(x-i-1)q(x-j)\\
- &\sum_{i,j\in\ZZ} p(i)p(j+1)q(x-i-1)q(x-j) - \sum_{i,j\in\ZZ} p(i+1)p(j)q(x-i)q(x-j-1)\\
= &2\sum_{i,j\in\ZZ} p(i)p(j)q(x-i)q(x-j) - 2\sum_{i,j\in\ZZ} p(i)p(j)q(x-i-1)q(x-j+1)\\
= &2\pr{\sum_{i\in\ZZ} p(i)q(x-i)}^2 - 2\pr{\sum_{i\in\ZZ} p(i)q(x-i-1)}\pr{\sum_{j\in\ZZ} p(j)q(x-j+1)}\\
= &2 (p*q)(x)^2 - 2(p*q)(x-1)\cdot(p*q)(x+1) . \qedhere
\end{align*}
\end{proof}

\begin{proof}[Proof of \Cref{prop:convolve_with_offset_m_lip_kernel}]
Since $g\coloneqq f*\1_{\cbr{0,1,...,m-1}}$ is weakly log-concave by \Cref{prop:discrete_weak_concave_after_convolve},
it suffices to prove that $g(n+1)g(n+m)\geq \beta\cdot g(n)g(n+m+1)$ for all $n\in\ZZ$. Fixing $n\in\ZZ$, this is equivalent to
\begin{equation*}
\pr{\sum_{i=0}^{m-1}{f(n+1-i)}} \pr{\sum_{i=0}^{m-1}{f(n+m-i)}}
\geq \beta \pr{\sum_{i=0}^{m-1}{f(n-i)}}\pr{\sum_{i=0}^{m-1}{f(n+m+1-i)}}.
\end{equation*}
Expanding, we get
\begin{equation} \label{ineq:before_split_by_ijsum}
\sum_{i=0}^{m-1}\sum_{j=0}^{m-1}{f(n+1-i)f(n+m-j)} 
\geq \beta \sum_{i=0}^{m-1}\sum_{j=0}^{m-1}{f(n-i)f(n+m+1-j)}.
\end{equation}
We present \eqref{ineq:before_split_by_ijsum} as the sum of $2m-1$ inequalities. Namely, we split the terms in each side according to the value of $r\coloneqq i+j\in \cbr{0,1,...,2m-2}$. When $r\leq m-1$, the corresponding terms on each side of \eqref{ineq:before_split_by_ijsum} are
\begin{equation} \label{ineq:after_split_small_r}
\sum_{i=0}^{r}{f(n+1-i)f(n+m-r+i)}
\geq \beta \sum_{i=0}^{r}{f(n-i)f(n+m+1-r+i)}
\end{equation}
and when $r\geq m-1$, the corresponding terms on each side of \eqref{ineq:before_split_by_ijsum} are
\begin{equation} \label{ineq:after_split_large_r}
\sum_{i=r-m+1}^{m-1}{f(n+1-i)f(n+m-r+i)}
\geq \beta \sum_{i=r-m+1}^{m-1}{f(n-i)f(n+m+1-r+i)}.
\end{equation}
Therefore it suffices to prove \eqref{ineq:after_split_small_r} for $r\leq m-1$ and \eqref{ineq:after_split_large_r} for $r>m-1$. We focus on \eqref{ineq:after_split_small_r}, as the two cases are symmetrical. Fix $0\leq r\leq m-1$, and let $F(i)\coloneqq f(n+1-i)f(n+m-r+i)$. We need to prove that
\begin{equation*}
\sum_{i=0}^{r} F(i) \geq \beta \sum_{i=1}^{r+1} F(i) ,
\end{equation*}
or after rearranging,
\begin{equation*}
F(0) \geq (\beta-1) \sum_{i=1}^{r} F(i) + \beta F(r+1).
\end{equation*}
Note that the sequence $F(0), F(1), ..., F(r+1)$ is decreasing, as for $0\leq i\leq r$ we have
\begin{equation*}
F(i) = f(n+1-i)f(n+m-r+i) \geq f(n-i)f(n+m-r+i+1) = F(i+1)
\end{equation*}
by \Cref{claim:weak_concave_over_distance}, since $n+m-r+i\geq n+1-i$.
Thus it suffices to prove that
\begin{equation*}
F(0) \geq (\beta-1) \sum_{i=1}^{r} F(0) + \beta F(r+1) ,
\end{equation*}
or after rearranging once more,
\begin{equation*}
F(0) \geq \frac{\beta}{1-r(\beta-1)} F(r+1).
\end{equation*}
Moreover, by \Cref{claim:strong_concave_over_distance} we have
\begin{equation*}
F(r+1) = f(n-r)f(n+m+1) \leq \alpha^{(n-r)-(n+1)} f(n+1)f(n+m-r) = \alpha^{-r-1} F(0),
\end{equation*}
since $n-r\leq n+1$.
Thus, it suffices to show that
\begin{equation*}
\frac{\beta}{1-r(\beta-1)} \leq \alpha^{r+1} ,
\end{equation*}
or equivalently,
\begin{equation} \label{ineq:lipschitz_beta_upper_bound}
\beta \leq \frac{r + 1}{r + \alpha^{-r-1}}.
\end{equation}
Observe that equality is obtained in \eqref{ineq:lipschitz_beta_upper_bound} when $r=m-1$ by definition of $\beta$. Since $r\leq m-1$, it suffices to show that $r\mapsto \frac{r + 1}{r + \alpha^{-r-1}}$ is decreasing in $r$. To verify this, we differentiate and get
\begin{equation*}
\frac{d}{dx} \pr{\frac{x+1}{x+\alpha^{-x-1}}} = \frac{\log\pr{\alpha^{x+1}}+1-\alpha^{x+1}}{\alpha^{x+1}\pr{x+\alpha^{-x-1}}^2} \leq 0 . \qedhere
\end{equation*}
\end{proof}

\begin{proof}[Proof of \Cref{prop:bounds_on_strong_log_conc_laws_lipschitz}]
Let $f:\ZZ\to[0,1]$ be the law of $X$. To prove \ref{item:lip_strong_stoch_dom} via \Cref{lem:likelyhood_dom_implies_stoch_dom}, we bound the successive quotients $f(n+1)/f(n)$. We may safely assume that all subsequent divisions by values of $f$ are well-defined, as the condition of \Cref{lem:likelyhood_dom_implies_stoch_dom} holds trivially outside the (contiguous) support of $f$.
By symmetry and $(\alpha,2M+1)$-strong log-concavity of $f$, we have
\begin{equation*}
f(M)^2 = f(-M)f(M) \geq \alpha f(-M-1)f(M+1) = \alpha f(M+1)^2,
\end{equation*}
which implies $f(M+1)/f(M) \leq \alpha^{-1/2}$. Moreover, we have $\frac{f(x+1)}{f(x)} \geq \frac{\alpha f(x+2M+2)}{f(x+2M+1)}$. Applying this iteratively starting from $x=M$ yields
\begin{equation*} 
\alpha^{-1/2} \geq \frac{f(M+1)}{f(M)} \geq \alpha \frac{f(3M+2)}{f(3M+1)} \geq \alpha^2 \frac{f(5M+3)}{f(5M+2)} \geq \cdots \geq \alpha^k \frac{f\pr{(2k+1)M+k+1}}{f\pr{(2k+1)M+k}},
\end{equation*}
which simplifies to
\begin{equation*}
\frac{f\pr{(2k+1)M+k+1}}{f\pr{(2k+1)M+k}} \leq \alpha^{-(2k+1)/2}.
\end{equation*}
By log-concavity of $f$, the ratio $f(n+1)/f(n)$ decreases in $n$. Thus, letting $k = \left\lfloor\frac{n-M}{2M+1}\right\rfloor$ so that $n \geq (2k+1)M+k$, monotonicity of this ratio guarantees
\begin{equation} \label{ineq:lip_localization_stoch_dom_helper_ineq}
\frac{f(n+1)}{f(n)} \leq \frac{f\pr{(2k+1)M+k+1}}{f\pr{(2k+1)M+k}} \leq \alpha^{-(2k+1)/2}.
\end{equation}
Let $Y=M+1+\abs{\mathcal{N}_{\ZZ}\pr{0,\frac{4M+2}{\log\alpha}}}$, whose law satisfies
\begin{equation*}
f_{Y}(n)\propto\begin{cases}
            0, & n<M+1\\
			1, & n=M+1\\
            2\alpha^{-(n-M-1)^2/(8M+4)}, & n>M+1
		 \end{cases}.
\end{equation*}
In comparison, the law of $\abs{X}$ is given by
\begin{equation*}
f_{\abs{X}}(n)=\begin{cases}
			f(0), & n=0\\
            2f(n), & n>0
		 \end{cases}.
\end{equation*}
For every $n\geq M+1$ in the support of $f$, \eqref{ineq:lip_localization_stoch_dom_helper_ineq} gives
\begin{equation*}
\frac{f_{Y}\pr{n+1}}{f_{Y}\pr{n}} \geq \alpha^{-\frac{2n-2M-1}{8M+4}}
\geq \alpha^{-\left\lfloor\frac{n-M}{2M+1}\right\rfloor-\frac{1}{2}} \geq \frac{f_{\abs{X}}\pr{n+1}}{f_{\abs{X}}\pr{n}}.
\end{equation*}
By \Cref{lem:likelyhood_dom_implies_stoch_dom} (noting that its condition degenerates in any other case), we deduce that $Y$ stochastically dominates $\abs{X}$. This verifies \ref{item:lip_strong_stoch_dom}, and also yields $\var{X} \leq \expected{Y^2}$.
Write $Y = M+1+Z$, where $Z = \abs{\mathcal{N}_{\ZZ}\pr{0,\sigma^2}}$ and $\sigma^2 = \frac{4M+2}{\log\alpha}$. Since $\expected{Z^2} \leq \sigma^2$ (see e.g. \cite{CKS20}), we get
\begin{equation*}
\expected{Y^2} \leq 2(M+1)^2 + 2\expected{Z^2}
\leq 2(M+1)^2 + \frac{8M+4}{\log\alpha},
\end{equation*}
establishing \ref{item:lip_strong_variance_bound}.
To prove \ref{item:lip_strong_pointwise_bound}, observe that for any $n > M$,
\begin{equation*}
\frac{f(n)}{f(M)} = \prod_{i=M}^{n-1} \frac{f(i+1)}{f(i)}.
\end{equation*}
Using the bound $\frac{f(i+1)}{f(i)} \leq \alpha^{-\frac{2i-2M-1}{8M+4}}$ established in \eqref{ineq:lip_localization_stoch_dom_helper_ineq}, we get
\begin{equation*}
\frac{f(n)}{f(M)} \leq \prod_{i=M}^{n-1} \alpha^{-\frac{2i-2M-1}{8M+4}} = \alpha^{-\frac{(n-M-1)^2 - 1}{8M+4}}.
\end{equation*}
Since $f(M) \leq f(0)$ by unimodality, the result follows.
\end{proof}

\subsection{Delocalization on recurrent trees}

Once again, the proof for delocalization of random Lipschitz functions follows the same structure as in the case of homomorphisms in \Cref{sec:delocalization}. We use a nearly identical definition of limited log-concavity, only amended to support functions whose domain is the entirety of $\ZZ$.

\begin{definition}
Given $\gamma,R>0$, let $\Phi^\gamma_R:(-R,R)\to\RR_+$ be defined by $\Phi^\gamma_R(x)=e^{\gamma/\pr{R-\abs{x}}}$. We say that $f:\ZZ\to\RRnonn$ is \emph{$\pr{R,\gamma}$-limited log-concave}, if it is weakly log-concave and
\begin{equation*}
f(n)^2 \leq \Phi^\gamma_R(2n)\cdot f(n-1)f(n+1)
\qquad\text{for all integers } n \in (-R/2,R/2).
\end{equation*}
\end{definition}
In other words, $f:\ZZ\to\RRnonn$ is $\pr{R,\gamma}$-limited log-concave whenever $g:2\ZZ\to\RRnonn$, defined by $g(2n)=f(n)$, is $\pr{R,\gamma}$-limited log-concave in the sense of \Cref{defn:hom_limited_log_concave}.

As before, the three core properties we require are conservation under products and under convolution with $\1_{\cbr{-M,...,M}}$, as well as distribution bounds.

\begin{proposition} \label{prop:product_lipschitz_limited_log_conc}
Let $f_1,f_2:\ZZ\to\RRnonn$ such that $f_i$ is $\pr{R_i,\gamma}$-limited log-concave. Then $f_1\cdot f_2$ is $\pr{R,\gamma}$-limited log-concave where $\frac{1}{R}=\frac{1}{R_1}+\frac{1}{R_2}$.
\end{proposition}

\begin{proposition} \label{prop:convolve_with_offset_m_lip_kernel_limited}
Let $f:\ZZ\to\RRnonn$ be symmetric around 0 and $\pr{R,\gamma}$-limited log-concave, with $\gamma\geq32M+16$. Then $f*\1_{\cbr{-M,...,M}}$ is $\pr{R+1,\gamma}$-limited log-concave.
\end{proposition}

\begin{proposition} \label{prop:bounds_on_limited_log_conc_laws_lipschitz}
Let $X$ be an integer-valued random variable whose law is symmetric around zero and $\pr{R,\gamma}$-limited log-concave with $\gamma, R>0$. Then the following hold:
\begin{enumerate}[(a)]
    \item $\abs{X}$ stochastically dominates a truncated discrete Gaussian $\abs{\mathcal{N}_{\ZZ\cap(-R/4,R/4)}\pr{0,\frac{R}{2\gamma}}}$;
    \item $\var{X} \geq \frac{R}{2\gamma e^4 \sqrt\pi}$ whenever $R \geq \max\pr{4\gamma, 64/\gamma}$; 
    \item $\prob{X = n}\geq e^{-\gamma n^2/R}\cdot \prob{X = 0}$ for all integers $-R/4 < n < R/4$;
    \item $\prob{X = n}\leq \frac{\max\pr{e^{\gamma},4}}{\sqrt{R}}$ for all $n\in\ZZ$.
\end{enumerate}
\end{proposition}

Out of the three, only \Cref{prop:convolve_with_offset_m_lip_kernel_limited} warrants a completely new proof, as the other two are essentially equivalent to their \Cref{sec:delocalization} counterparts (\Cref{prop:prod_limited_log_conc,prop:bounds_on_limited_log_conc_laws}). We defer the proof of \Cref{prop:convolve_with_offset_m_lip_kernel_limited} to \Cref{ssec:proof_of_lipschitz_limited_convolve}.
Note that the requirement $\gamma\geq 32M+16$ is not the tightest possible, but enables simpler algebra below.

\begin{theorem} \label{thm:lip_root_law_is_limited_log_conc}
Let $M\geq1$ and let $(T,v^*)$ be a finite rooted tree. Then $f_{T,M-\mathrm{Lip}}$ is $\pr{R,\gamma}$-limited log-concave for $R=\mathcal{R}_T$ and $\gamma=32M+16$.
\end{theorem}

\begin{proof}
Follows inductively from \Cref{lem:resistance_flow} and \Cref{prop:product_lipschitz_limited_log_conc,prop:convolve_with_offset_m_lip_kernel_limited}.
\end{proof}

\begin{corollary} \label{cor:lower_bounds_on_lip_root_laws}
Let $(T,v^*)$ be an even finite rooted tree, and let $M$ be a positive integer. Let $\phi:V(T)\to\ZZ$ be a uniformly random $M$-Lipschitz function satisfying $\phi\equiv0$ on the leaves of $T$. With $R=\mathcal{R}_T$, the following hold:
\begin{enumerate}[(a)]
    \item $\abs{\phi\pr{v^*}}$ stochastically dominates $\abs{\mathcal{N}_{\ZZ\cap(-R/4,R/4)}\pr{0,\frac{R}{64M+32}}}$;
    \item $\var\pr{\phi\pr{v^*}} \geq \floor{\frac{R}{(64M+32)e^4 \sqrt{\pi}}}$;
    \item $\prob{\phi\pr{v^*} = n}\geq e^{-(32M+16)n^2/R}\cdot \prob{\phi\pr{v^*} = 0}$ for all integers $-R/4 < n < R/4$;
    \item $\prob{\phi\pr{v^*} = n}\leq \frac{e^{32M+16}}{\sqrt{R}}$ for all $n$.
\end{enumerate}
\end{corollary}

\begin{proof}
Follows directly from \Cref{prop:bounds_on_limited_log_conc_laws_lipschitz} and \Cref{thm:lip_root_law_is_limited_log_conc}.
\end{proof}

For recurrent trees, which have infinite resistance, taking $R\to\infty$ in the distribution bounds above readily yields delocalization of Lipschitz functions (as in the proof of \Cref{thm:hom_deloc_on_rec_tree}).

\begin{theorem} \label{thm:lip_deloc_on_rec_tree}
Let $T$ be a recurrent locally finite tree. Then integer-valued $M$-Lipschitz functions on $T$ are delocalized for all integers $M\geq 1$.
\end{theorem}

Combining \Cref{thm:lip_loc_on_tr_tree,thm:lip_deloc_on_rec_tree}, we obtain the complete transient-recurrent dichotomy stated in \Cref{thm:equivalence_lip}.

\subsection{Proof of \Cref{prop:convolve_with_offset_m_lip_kernel_limited}}
\label{ssec:proof_of_lipschitz_limited_convolve}
In this subsection, we prove that limited log-concavity is preserved under convolution by $\1_{\cbr{-M,...,M}}$, provided $\gamma\geq 32M+16$. Let $f:\ZZ\to\RRnonn$ be $(R,\gamma)$-limited log-concave and symmetric. We denote $g\coloneqq f*\1_{\cbr{-M,...,M}}$ and prove that
\begin{equation*} 
g(n)^2\leq \Phi^\gamma_{R+1}(2n) \cdot g(n-1)g(n+1) 
\qquad\text{for all integers } n \in \pr{-\frac{R+1}{2},\frac{R+1}{2}}.
\end{equation*}
This expands to
\begin{equation} \label{ineq:lipschitz_limited_convolve_main_ineq}
\sum_{i=-M}^{M}\sum_{j=-M}^{M}f(n+i)f(n+j) \leq \Phi^\gamma_{R+1}(2n) \sum_{i=-M}^{M}\sum_{j=-M}^{M}f(n+i-1)f(n+j+1) .
\end{equation}
Mirroring the approach used in \Cref{prop:convolve_with_offset_m_lip_kernel}, we split this inequality into several parts based on the value of $i+j$. We fix an integer $n\in\pr{-\frac{R+1}{2},\frac{R+1}{2}}$, and for every integer $S\in\br{-2M,2M}$ we denote
\begin{align*}
L_S &\coloneqq \sum_{\stackrel{-M\leq i,j\leq M}{i+j=S}}f(n+i)f(n+j) , \\
R_S &\coloneqq \Phi^\gamma_{R+1}(2n) \sum_{\stackrel{-M\leq i,j\leq M}{i+j=S}}f(n+i-1)f(n+j+1) .
\end{align*}
With this notation, the desired inequality \eqref{ineq:lipschitz_limited_convolve_main_ineq} simply translates to
\begin{equation*}
\sum_{S=-2M}^{S=2M} L_S \leq \sum_{S=-2M}^{S=2M} R_S.
\end{equation*}

\begin{claim} \label{claim:lipschitz_deloc_helper_middle}
Suppose that $\abs{n-M}< \frac{R}{2}$ and $\abs{n+M}<\frac{R}{2}$. Then for all $-2M<S<2M$ we have
\begin{equation*}
L_S \leq R_S.
\end{equation*}
\end{claim}

\begin{claim} \label{claim:lipschitz_deloc_helper_edges}
Suppose that $\abs{n-M}< \frac{R}{2}$ and $\abs{n+M}<\frac{R}{2}$. Then
\begin{align*}
\quad L_{2M}+L_{2M-1} &\leq R_{2M}+R_{2M-1}, \\
\quad L_{-2M}+L_{-2M+1} &\leq R_{-2M}+R_{-2M+1}.
\end{align*}
\end{claim}

\begin{claim} \label{claim:lipschitz_deloc_helper_small_r}
Suppose that $\abs{n-M}\geq \frac{R}{2}$ or $\abs{n+M}\geq \frac{R}{2}$. Then
\begin{equation*}
\sum_{S=-2M}^{S=2M} L_S \leq \sum_{S=-2M}^{S=2M} R_S.
\end{equation*}
\end{claim}

Equipped with these three claims, the rest of the proof is straightforward:
If $\abs{n-M}< \frac R2$ and $\abs{n+M}< \frac R2$ then \Cref{claim:lipschitz_deloc_helper_middle,claim:lipschitz_deloc_helper_edges} give
\begin{multline*}
\pr{L_{-2M} + L_{-2M+1}} + L_{-2M+2}+\cdots+\pr{L_{2M-1} + L_{2M}} \leq \\
\leq \pr{R_{-2M} + R_{-2M+1}} + R_{-2M+2}+\cdots+\pr{R_{2M-1} + R_{2M}},
\end{multline*}and otherwise $\sum_SL_S\leq \sum_SR_S$ follows from \Cref{claim:lipschitz_deloc_helper_small_r}. We finish by proving these claims. The high-level strategy used to prove each claim is similar to \Cref{prop:conv_limited_log_conc}: We use limited log-concavity to eliminate the $f(\cdot)$ terms, and thus obtain elementary inequalities on $R,\gamma,n,M,S$.

\begin{proof} [Proof of \Cref{claim:lipschitz_deloc_helper_middle}]
By convexity of $x\mapsto\abs{x}$, we have $\abs{n+k}<\frac{R}{2}$ for all $-M\leq k\leq M$. By symmetry, we may assume that $S\geq 0$. We need to prove that
\begin{equation*}
\sum_{i=S-M}^{M}f(n+i)f(n+S-i) \leq \Phi^\gamma_{R+1}(2n) \sum_{i=S-M}^{M}f(n+i-1)f(n+S+1-i)
\end{equation*}
or equivalently,
\begin{multline}\label{ineq:deloc_helper_middle_coeffs}
0 \leq \sum_{i=S-M}^{M-1}\pr{\Phi^\gamma_{R+1}(2n)-1}f(n+i)f(n+S-i) + \\
+ \Phi^\gamma_{R+1}(2n)f(n+S-M-1)f(n+M+1) - f(n+M)f(n+S-M).
\end{multline}
Since $f$ is weakly log-concave, \Cref{claim:weak_concave_over_distance_sym} gives for $S-M\leq i \leq M-1$ that
\begin{equation}\label{ineq:deloc_helper_middle_termwise_replacement_1}
\pr{\Phi^\gamma_{R+1}(2n)-1}f(n+i)f(n+S-i) \geq \pr{\Phi^\gamma_{R+1}(2n)-1}f(n+M)f(n+S-M)
\end{equation}
Moreover, applying $\pr{R,\gamma}$-limited log-concavity at $n+k$ for $S-M\leq k\leq M$, we have
\begin{equation}\label{ineq:deloc_helper_middle_termwise_replacement_2}
f(n+M+1)f(n+S-M-1)\geq \frac{f(n+M)f(n+S-M)}{\prod_{S-M}^{M}\Phi^\gamma_R\pr{2n+2k}}.
\end{equation}
(deriving this requires safely dividing by intermediate values of $f$ which must be positive).
Substituting \eqref{ineq:deloc_helper_middle_termwise_replacement_1} and \eqref{ineq:deloc_helper_middle_termwise_replacement_2} into \eqref{ineq:deloc_helper_middle_coeffs}, and dividing by $\Phi^\gamma_{R+1}(2n)f(n+M)f(n+S-M)$, it suffices to prove that
\begin{equation} \label{ineq:deloc_helper_middle_fless}
0 \leq 2M-S
+ \frac{1}{\prod_{S-M}^{M}\Phi^\gamma_R\pr{2n+2k}} - \frac{2M-S+1}{\Phi^\gamma_{R+1}(2n)}.
\end{equation}
Note that $\abs{n+k}\leq|n|+M$ for all $k\in[S-M,M]$, and therefore $\Phi^\gamma_R\pr{2n+2k}\leq\Phi^\gamma_R\pr{|2n|+2M}$. Thus it suffices to prove that
\begin{equation*} 
0 \leq 2M-S
+ \frac{1}{\pr{\Phi^\gamma_R\pr{|2n|+2M}}^{2M-S+1}} - \frac{2M-S+1}{\Phi^\gamma_{R+1}(2n)},
\end{equation*}
or more explicitly, that
\begin{equation*}
F(R) \coloneqq 2M-S
+ \exp\pr{-\frac{(2M-S+1)\gamma}{R-\abs{2n}-2M}}
- \pr{2M-S+1}\exp\pr{-\frac{\gamma}{R+1-\abs{2n}}}
\geq 0 .
\end{equation*}
At this point we may assume that $n\geq0$ and thus eliminate the absolute value. Note that $\lim_{R\to\infty}F(R)=0$, so it suffices to show that $F$ is decreasing.
Differentiating, we need to show that
\begin{equation*}
F'(R)=\frac{\pr{2M-S+1}\gamma}{\pr{R-2n-2M}^2}\exp\pr{-\frac{\pr{2M-S+1}\gamma}{R-2n-2M}}
-\frac{\pr{2M-S+1}\gamma}{\pr{R+1-2n}^2}\exp\pr{-\frac{\gamma}{R+1-2n}} \leq 0,
\end{equation*}
or after rearranging and taking a square root,
\begin{equation} \label{ineq:final_part_of_lipschitz_deloc_helper_middle}
\exp\pr{\frac{\pr{2M-S+1}\gamma}{2\pr{R-2n-2M}}
-\frac{\gamma}{2\pr{R+1-2n}}}
\geq
\frac{R+1-2n}{R-2n-2M}.
\end{equation}
To verify \eqref{ineq:final_part_of_lipschitz_deloc_helper_middle} we apply $e^x\geq x+1$ and find that
\begin{multline*}
\exp\pr{\frac{\pr{2M-S+1}\gamma}{2\pr{R-2n-2M}}-\frac{\gamma}{2\pr{R+1-2n}}}
\geq
1 + \frac{\pr{2M-S+1}\gamma}{2\pr{R-2n-2M}}-\frac{\gamma}{2\pr{R+1-2n}} \geq \\
\geq 1 + \frac{\pr{2M-S+1}\gamma}{2\pr{R-2n-2M}}-\frac{\gamma}{2\pr{R-2n-2M}}
= \frac{R-2n-2M+\pr{M-\frac{1}{2}S}\gamma}{R-2n-2M} 
\geq \frac{R+1-2n}{R-2n-2M},
\end{multline*}
where the final transition follows from the inequalities $S\leq 2M-1$ and $\gamma\geq4M+2$.
\end{proof}

\begin{proof} [Proof of \Cref{claim:lipschitz_deloc_helper_edges}]
By symmetry, it suffices to prove the first part of the claim.
Suppose first that $|n+M| < |n|$. Then
\begin{equation*}
\Phi^\gamma_R(2n+2M)=\exp\pr{\frac{\gamma}{R-|2n+2M|}}
< \exp\pr{\frac{\gamma}{R+1-|2n|}} = \Phi^\gamma_{R+1}(2n).
\end{equation*}
Now $\pr{R,\gamma}$-limited log-concavity gives
\begin{multline} \label{ineq:edges_easy_case}
L_{2M}=f\pr{n+M}^2\leq\Phi^\gamma_R(2n+2M)f(n+M-1)f(n+M+1) \leq \\
\leq \Phi^\gamma_{R+1}(2n)f(n+M-1)f(n+M+1) = R_{2M}.
\end{multline}
Combining \eqref{ineq:edges_easy_case} with the fact that $L_{2M-1}\leq R_{2M-1}$ by \Cref{claim:lipschitz_deloc_helper_middle}, we are done in this case.

Therefore we may assume that $|n+M|\geq|n|$. This in particular implies that $n+M>0$.
Explicitly, we need to prove that
\begin{multline} \label{ineq:edges_hard_case}
f\pr{n+M}^2 + 2f\pr{n+M-1}f\pr{n+M}
\leq \Phi^\gamma_{R+1}(2n)\cdot(f(n+M-1)f(n+M+1)+\\
+f(n+M-2)f(n+M+1)+f(n+M-1)f(n+M)).
\end{multline}
By $\pr{R,\gamma}$-limited log-concavity of $f$ we have the inequalities
\begin{align}
f(n+M-1)f(n+M+1) &\geq \frac{f\pr{n+M}^2}{\Phi^\gamma_{R}(2n+2M)}, \label{ineq:edges_hard_case_bound_1}
\\
f(n+M-2)f(n+M+1) &\geq \frac{f(n+M-1)f(n+M)}{\Phi^\gamma_{R}(2n+2M)\Phi^\gamma_{R}(2n+2M-2)}. \label{ineq:edges_hard_case_bound_2}
\end{align}
Since $\Phi^\gamma_R$ is increasing in $\RRnonn$, we can simplify \eqref{ineq:edges_hard_case_bound_2} to
\begin{equation} \label{ineq:edges_hard_case_bound_2_simple}
f(n+M-2)f(n+M+1) \geq \frac{f(n+M-1)f(n+M)}{\Phi^\gamma_{R}(2n+2M)^2}. 
\end{equation}
Substituting \eqref{ineq:edges_hard_case_bound_1} and \eqref{ineq:edges_hard_case_bound_2_simple} into \eqref{ineq:edges_hard_case}, it suffices to show that
\begin{multline*}
f\pr{n+M}^2 + 2f\pr{n+M-1}f\pr{n+M} \leq \\
\leq \Phi^\gamma_{R+1}(2n)\cdot\pr{\frac{f\pr{n+M}^2}{\Phi^\gamma_{R}(2n+2M)}
+\frac{f(n+M-1)f(n+M)}{\Phi^\gamma_{R}(2n+2M)^2}+f(n+M-1)f(n+M)}
\end{multline*}
or after dividing by $\Phi^\gamma_{R+1}(2n) f(n+M)$ and rearranging,
\begin{multline*}
\pr{\frac{1}{\Phi^\gamma_{R+1}(2n)} - \frac{1}{\Phi^\gamma_{R}(2n+2M)}}
f(n+M)
\leq \\ \leq
\pr{\frac{1}{\Phi^\gamma_{R}(2n+2M)^2}+1
-\frac{2}{\Phi^\gamma_{R+1}(2n)}}
f(n+M-1).
\end{multline*}
Since $f(n+M)\leq f(n+M-1)$ by unimodality around zero, it suffices to show that
\begin{equation} \label{ineq:edges_hard_case_no_f}
\frac{1}{\Phi^\gamma_{R+1}(2n)} - \frac{1}{\Phi^\gamma_{R}(2n+2M)}
\leq
\frac{1}{\Phi^\gamma_{R}(2n+2M)^2}+1
-\frac{2}{\Phi^\gamma_{R+1}(2n)}.
\end{equation}
Writing $u=\pr{R+1-|2n|}/\gamma$ and $v=\pr{R-2n-2M}/\gamma$, \eqref{ineq:edges_hard_case_no_f} simply translates to 
\begin{equation} \label{ineq:3u_1_v_v2}
3e^{-1/u}\leq 1+e^{-1/v}+e^{-2/v} .
\end{equation}
Note that $u>v$ since $n+M \geq |n|$, and that $u-v \leq (2M+1)/\gamma \leq 1/16$. We proceed to show that \eqref{ineq:3u_1_v_v2} holds purely as a consequence of the fact that $0 < u-v \leq 1/16$ and $v>0$. We divide into three cases.

\paragraph{Case 1: $u < 1/2$.} In this case, we simply have
\begin{equation*}
3e^{-1/u} < 3e^{-2}<1<1+e^{-1/v}+e^{-2/v} .
\end{equation*}

\paragraph{Case 2: $v > 2$.} 
We apply the Taylor series bounds
\begin{align*}
e^{-1/u}&\leq 1 - \frac{1}{u} + \frac{1}{2u^2} ,\\
e^{-1/v}&\geq 1 - \frac{1}{v} + \frac{1}{2v^2} - \frac{1}{6v^3} ,\\
e^{-2/v}&\geq 1 - \frac{2}{v} + \frac{4}{2v^2} - \frac{8}{6v^3} .
\end{align*}
Substituting these into \eqref{ineq:3u_1_v_v2}, it suffices to prove that
\begin{equation*}
3 - \frac{3}{u} + \frac{3}{2u^2} \leq 3 - \frac{3}{v} + \frac{5}{2v^2} - \frac{9}{6v^3} 
\end{equation*}
or after rearranging,
\begin{equation*}
\frac{3(u-v)}{uv} + \frac{3\pr{v^2-u^2}}{2u^2v^2} \leq \frac{1}{v^2} - \frac{3}{2v^3} .
\end{equation*}
Bounding above the terms in the left-hand side by $\frac{3}{16v^2}$ and by $0$ respectively, we find indeed that
\begin{equation*}
\frac{3(u-v)}{uv} + \frac{3\pr{v^2-u^2}}{2u^2v^2}
< \frac{3}{16v^2}
< \frac{1}{v^2} - \frac{3}{2v^3}
\end{equation*}
since $v > 2$.
\paragraph{Case 3: $u \geq 1/2$ and $v\leq 2$.} Using that $u-v \le 1/16$ and $u \ge 1/2$, we see that
\begin{equation*}
\frac{v}{u} = 1 - \frac{u-v}{u}
\geq \frac{7}{8}.
\end{equation*}
Replacing $e^{-1/u}$ by $e^{-7/8v}$, it suffices to show that
\begin{equation*}
3e^{-7/8v} \leq 1+e^{-1/v}+e^{-2/v} .
\end{equation*}
This holds as the smallest positive solution to $3t^{7/8}=1+t+t^2$ is $t_0\approx 0.684$, while the largest value $e^{-1/v}$ can take in this case is $e^{-1/2} \approx 0.607$.
\end{proof}

\begin{proof} [Proof of \Cref{claim:lipschitz_deloc_helper_small_r}]
Since $R\leq |2n|+2M$ by the triangle inequality, we have
\begin{equation} \label{ineq:small_r_phi_bound}
\Phi^\gamma_{R+1}(2n) = \exp\pr{\frac{\gamma}{R+1-|2n|}} \geq \exp\pr{\frac{\gamma}{2M+1}} > 4 .
\end{equation}
We bound $\sum L_S$ from above by the sum of four terms:
\begin{align} 
\sum_{S=-2M}^{S=2M} L_S=\sum_{i=-M}^{M}\sum_{j=-M}^{M}f(n+i)f(n+j) &\leq
\sum_{i=-M}^{M-1}\sum_{j=-M+1}^{M}f(n+i)f(n+j) \label{eq:four_square_cover_square_1}
\\ &+ \sum_{i=-M+1}^{M}\sum_{j=-M}^{M-1}f(n+i)f(n+j) \label{eq:four_square_cover_square_2}
\\ &+ f(n-M)f(n-M) \label{eq:four_square_cover_corner_1}
\\ &+ f(n+M)f(n+M) . \label{eq:four_square_cover_corner_2}
\end{align}
This bound holds as every summand in the left-hand side appears once or twice in the right-hand side.
We prove that each of the four terms is individually bounded above by $\frac{1}{4}\sum_S R_S$.

Observe that the first term \eqref{eq:four_square_cover_square_1} is a subsum of 
\begin{equation*}
\sum_{i=-M}^{M}\sum_{j=-M}^{M}f(n+i-1)f(n+j+1) = \frac{1}{\Phi^\gamma_{R+1}(2n)}\sum R_S ,
\end{equation*}
which is smaller than $\frac{1}{4}\sum R_S$ due to \eqref{ineq:small_r_phi_bound}.
The second term \eqref{eq:four_square_cover_square_2} equals the first by symmetry, so it is also bounded by $\frac{1}{4}\sum R_S$.
To bound the third term \eqref{eq:four_square_cover_corner_1}, we divide into three cases. The bound on the fourth term \eqref{eq:four_square_cover_corner_2} is obtained symmetrically.

\paragraph{Case 1: $\abs{n-M}\geq\abs{n}$.} In this case, unimodularity of $f$ gives
\begin{equation*}
f\pr{n-M}^2 \leq f\pr{n}^2 \leq \frac{1}{\Phi^\gamma_{R+1}(2n)}\sum R_S
\end{equation*}
as $\Phi^\gamma_{R+1}(2n)f\pr{n}^2$ is one of the summands in the sum that defines $R_0$. Once again, \eqref{ineq:small_r_phi_bound} gives
\begin{equation*}
\frac{1}{\Phi^\gamma_{R+1}(2n)}\sum R_S \leq \frac{1}{4}\sum R_S .
\end{equation*}

\paragraph{Case 2: $\abs{n-M}<\abs{n}$, and $n<M$.}
The former condition also implies that $n>0$. Letting $i=M-2n+1$ and $j=M-2n-1$, we find that $\Phi^\gamma_{R+1}(2n)f\pr{-n+M}^2$ is one of the summands in the sum that defines $R_{2M-4n}$, and so
\begin{equation*}
f\pr{n-M}^2 = f\pr{-n+M}^2 \leq \frac{1}{\Phi^\gamma_{R+1}(2n)}\sum R_S
\leq \frac{1}{4}\sum R_S .
\end{equation*}

\paragraph{Case 3: $\abs{n-M}<\abs{n}$, and $n\geq M$.}
In this case,
\begin{equation*}
\abs{n-M}<n-\frac{1}{2}<\frac{R}{2}
\end{equation*}
so that $\pr{R,\gamma}$-limited log-concavity of $f$ gives
\begin{equation*}
f\pr{n-M}^2 \leq \Phi^\gamma_{R}(2n-2M) f\pr{n-M-1} f\pr{n-M+1}
\leq \frac{\Phi^\gamma_{R}(2n-2M)}{\Phi^\gamma_{R+1}(2n)} \sum R_S .
\end{equation*}
It remains to verify that $\Phi^\gamma_{R+1}(2n) \geq 4\cdot\Phi^\gamma_{R}(2n-2M)$. Taking logarithms, $R\leq 2M+2n$ yields
\begin{equation*}
\frac{\gamma}{R+1-2n} - \frac{\gamma}{R-2n+2M} = \frac{\pr{2M-1}\gamma}{\pr{R+1-2n}\pr{R-2n+2M}}
\geq \frac{\pr{2M-1}\cdot\pr{32M+16}}{\pr{2M+1}\cdot \pr{4M}} > \log(4) .
\end{equation*}
This completes the proof of the claim.
\end{proof}
Having established \Cref{claim:lipschitz_deloc_helper_middle,claim:lipschitz_deloc_helper_edges,claim:lipschitz_deloc_helper_small_r}, the proof of \Cref{prop:convolve_with_offset_m_lip_kernel_limited} is complete.

%% file: counterexample.tex
\section{Non-tree counterexamples} \label{sec:counterexample}

In this section, we provide examples to show that the characterization of localization-delocalization in terms of transience-recurrence, established in \Cref{thm:equivalence_hom,thm:equivalence_lip}, does not extend from trees to the class of locally finite graphs. In fact, we show that this fails in both directions. This is true both for random homomorphisms as well as integer-valued Lipschitz functions.

We begin with known results which demonstrate that recurrent graphs may exhibit localization for both of our height function models. The first part of the following theorem is a particular case of a result from~\cite{Pel17}, and the second part is a special case of a result from~\cite{PS20}.

\begin{theorem}[\cite{Pel17,PS20}]
Let $G=\ZZ^2 \times \{0,1\}^d$ where $d$ is a positive integer.
\begin{enumerate}[(a)]
    \item Homomorphisms and 1-Lipschitz functions are localized on $G$ whenever $d$ is large enough.
    \item For every $M>1$ there exists $d_0$ such that $M$-Lipschitz functions are localized on $G$ whenever $d \geq d_0$.
\end{enumerate}
\end{theorem}
These results hold more generally when $G=\ZZ^{d_1} \times \{0,1\}^{d_2}$ with $d_1 \geq 2$ and when $d=d_1+d_2$ is large as in the theorem. However, these graphs are recurrent only when $d_1=2$.

The second counterexample shows that random homomorphisms may be delocalized on general transient graphs. This example can be tuned to cover the case of integer-valued Lipschitz functions as well (see \Cref{remark:lip_ctrex}).

\begin{proposition}
There exists a transient locally finite graph on which homomorphisms are delocalized.
\end{proposition}

\begin{proof}
We start with an infinite binary tree. Denote by $A_n$ the set of vertices at depth $n\geq0$, so that $\abs{A_n}=2^n$.
We extend the graph as follows: For every $n\geq0$, add $2^{n+2}$ vertices to the graph, and connect each of them to each of the vertices in $A_n$. We denote this new set of vertices by $B_n$.

Informally, the presence of $B_n$ makes it highly likely that vertices in $A_n$ receive the same value under a uniformly random homomorphism with an appropriate boundary condition.
Indeed, since there is a complete bipartite subgraph connecting $A_n$ and $B_n$, any homomorphism must either be constant on $A_n$ or on $B_n$ for each $n\geq 0$. Since $\abs{B_n}$ was chosen sufficiently large, this causes the number of homomorphisms that are constant on $A_n$ to be a significant majority.
Conditioned on the event that a random homomorphism takes constant values on each of $A_0,A_1,A_2,A_3,...$, these values are distributed as a simple random walk, whose behavior is delocalized.
We proceed to formalize this idea.

\begin{claim} \label{claim:ctrex_layer_const_step}
Let $N>0$, and denote $\Lambda=\bigcup_{k=0}^{N} \pr{A_k \cup B_k}$. Let $\phi: \Lambda \to \ZZ$ be a uniformly sampled homomorphism satisfying the boundary condition $\phi|_{A_N} \equiv 0$. Then for all $1\leq n\leq N-1$, we have
\begin{equation*}
\prob{\phi\textrm{ is constant on }A_n\mid \phi\textrm{ is constant on }A_{n+1}}
\geq 1-2^{-2^n} .
\end{equation*}
\end{claim}
\begin{proof}
It suffices to prove that this holds when also conditioning on the heights of all vertices in $A_{n+1}$ and below.
In other words, showing that for every homomorphism $\psi:\bigcup_{k=n+1}^N \pr{A_k \cup B_k}\to\ZZ$ which is constant on $A_{n+1}$, we have
\begin{equation*}
\prob{\phi\textrm{ is constant on }A_n
\mid \phi\equiv \psi\textrm{ on }\bigcup_{k=n+1}^N \pr{A_k \cup B_k}}
\geq 1-2^{-2^n} .
\end{equation*}
Among homomorphisms $\phi$ that agree with $\psi$, we say that those that are also constant on $A_n$ are good, and the rest are bad.
Observe that there are at least $2^{\sum_{k=0}^{n}{\abs{B_k}}}$ good homomorphisms:
Indeed, if $\psi\mid_{A_{n+1}}\equiv a$, we may for example set $\psi\mid_{A_{n+1-k}}\equiv a+k$ and freely select heights $a+k\pm 1$ for each vertex in $B_{n+1-k}$, for every $1\leq k\leq n+1$.

On the other hand, there are at most $2^{\sum_{k=0}^n{\abs{A_k}}+\sum_{k=0}^{n-1}{\abs{B_k}}}$ bad homomorphisms:
Indeed, we may sequentially choose the height of every vertex in $A_n$, then $B_n$, $A_{n-1}, B_{n-1}$ and so forth.
At each step there are at most two valid options, while for vertices in $B_{n}$ there is only one valid option if heights are nonconstant on $A_n$.
In total, we get
\begin{multline*}
\prob{\phi\textrm{ is not constant on }A_n
\mid \phi\equiv \psi\textrm{ on }\bigcup_{k=n+1}^N \pr{A_k \cup B_k}} \leq \\
\leq \frac{2^{\sum_{k=0}^n{\abs{A_k}}+\sum_{k=0}^{n-1}{\abs{B_k}}}}{2^{\sum_{k=0}^{n}{\abs{B_k}}}}
= 2^{-\abs{B_n}+2^{n+1}-1} < 2^{-2^n}. \qedhere
\end{multline*}
\end{proof}
Fix $N>0$, and let $\phi: \Lambda \to \ZZ$ be sampled as in \Cref{claim:ctrex_layer_const_step}. Also denote $m=\floor{\log N}$. Let $E$ be the event that $\phi$ is constant on $A_n$ for all $m\leq n\leq N$. Note that $\phi$ is always constant on the boundary $A_N$. Applying the claim multiple times, we get as $N\to\infty$ that
\begin{equation*}
\prob{E}
\geq \prod_{n=m}^{N-1} \pr{1-2^{-2^n}}
\geq \prod_{n=m}^{\infty} \pr{1-2^{-2^n}}
\to 1.
\end{equation*}
Conditioned on $E$, the heights
\begin{equation*}
\pr{a_N, a_{N-1},...,a_m}
\coloneqq \pr{\phi|_{A_N}, \phi|_{A_{N-1}}, ...,\phi|_{A_m}}
\end{equation*}
follow the distribution of a simple random walk starting at zero.
Equivalently, the number of homomorphisms that generate any particular sequence $\pr{a_N, a_{N-1},...,a_m}$ is the same for any such valid sequence.
Indeed, given $\pr{a'_k}_{k=m}^N$ and $\pr{a''_k}_{k=m}^N$, there is a simple bijection between the sets of homomorphisms that produce them: Namely, adding $a''_{\max(k,m)}-a'_{\max(k,m)}$ to all heights in $A_k\cup B_k$ for all $0\leq k\leq N$.

In particular, the height at depth $m$ is conditionally distributed as $2 \cdot\mathrm{Bin}(N-m,1/2)-(N-m)$, which implies that as $N\to\infty$,
\begin{equation*}
\prob{\abs{\phi|_{A_m}} \geq \frac{1}{m}\sqrt{N-m} \mid E} \to 1.
\end{equation*}
Since $\phi|_{A_m}$ differs from the height at the root by at most $m$, we get
\begin{equation*}
\prob{\abs{\phi\pr{\mathrm{root}}} \geq \frac{1}{m}\sqrt{N-m} - m \mid E} \to 1.
\end{equation*}
Because $\prob{E}\to 1$, we in fact have
\begin{equation*}
\prob{\abs{\phi\pr{\mathrm{root}}} \geq \frac{1}{m}\sqrt{N-m} - m} \to 1.
\end{equation*}
Finally, since $\frac{1}{m}\sqrt{N-m} - m=\omega(1)$ we deduce that homomorphisms are delocalized on this graph. However, this graph is transient as it contains a transient subgraph, namely the infinite binary tree.

\end{proof}

\begin{remark} \label{remark:lip_ctrex}
By sufficiently increasing $\abs{B_n}$, one may obtain a (transient) graph on which random integer-valued $M$-Lipschitz functions are delocalized: Indeed, if an $M$-Lipschitz function is constant over $A_n$ then there are $\pr{2M+1}^{\abs{B_n}}$ possible combinations for the heights of $B_n$, and in any other case there are at most $\pr{2M}^{\abs{B_n}}$ such combinations. If $\abs{B_n}$ is chosen sufficiently large, the discrepancy between the two quantities may be utilized as in the proof above to establish delocalization of Lipschitz functions on this graph.
\end{remark}